\documentclass[11pt,a4paper]{article}
\usepackage{a4wide}
\usepackage{amsmath,amsthm,amssymb}

\usepackage{graphicx}
\usepackage{psfrag}

\newtheorem{theorem}{THEOREM}[section]
\newtheorem{lemma}[theorem]{LEMMA}
\newtheorem{corollary}[theorem]{COROLLARY}
\newtheorem{proposition}[theorem]{PROPOSITION}

\theoremstyle{definition}
\newtheorem{definition}[theorem]{Definition}
\newtheorem{construction}[theorem]{Construction}
\newtheorem{example}[theorem]{Example}
\newtheorem{remark}[theorem]{Remark}
\newtheorem{question}[theorem]{Question}
\newtheorem{problem}[theorem]{Problem}

\newcommand{\HoSi}{{\sf HoSi}} 

\newcommand{\norml}{\vartriangleleft}

\DeclareMathOperator{\Aut}{Aut}

\newcommand{\PG}{\mathop{\mathrm{PG}}}
\newcommand{\PSU}{\mathop{\mathrm{PSU}}}
\newcommand{\AGL}{\mathop{\mathrm{AGL}}}
\newcommand{\AG}{\mathop{\mathrm{AG}}}
\newcommand{\GF}{\mathop{\mathrm{GF}}}

\newcommand{\la}{\langle}
\newcommand{\ra}{\rangle}
\DeclareMathOperator{\Wr}{wr}
\DeclareMathOperator{\ssg}{ssg}

\def\dotcup{\DOTSB\mathop{\overset{\textstyle.}\cup}}

\title{Quotients of incidence geometries}
\author{Philippe Cara\footnote{The first author was partially
    supported by grant 15.263.08 of the ``Fonds voor Wetenschappelijk
    Onderzoek-Vlaanderen''.}  \\ 
Department of Mathematics\\
Vrije Universiteit Brussel\\
Pleinlaan 2\\
B-1050 Brussel\\
BELGIUM\\
pcara@vub.ac.be
\and Alice Devillers, Michael Giudici and Cheryl E.~Praeger\footnote{The paper forms
    part of Australian Research Council Discovery grant DP0770915
    which includes the Australian Research Fellowship of the third
    author. The second author is supported within the Australian
    Research Council Federation Fellowship FF0776186 project of the
    fourth author. The fourth author is
    supported by Australian Research Council Federation Fellowship
    FF0776186.} \\
School of Mathematics and Statistics\\
The University of Western Australia\\
35 Stirling Highway\\
Crawley WA 6009\\
AUSTRALIA\\ alice.devillers@uwa.edu.au, michael.giudici@uwa.edu.au, \\ cheryl.praeger@uwa.edu.au}

\begin{document}
\maketitle
\begin{abstract}
We develop a theory for quotients of geometries and obtain sufficient
conditions for the quotient of a geometry to be a geometry. These
conditions are compared with earlier work on quotients, in particular
by Pasini and Tits. We also explore geometric properties such as connectivity, firmness and transitivity conditions to determine when they are preserved under the quotienting operation. We show that the class of coset pregeometries, which contains all flag-transitive geometries, is closed under an appropriate quotienting operation.
\end{abstract}


\section{Introduction}
\label{sec:intro}

A common technique when studying a class of mathematical objects is to
undertake a quotienting process and reduce the problem to studying the
``basic'' objects in the class, these being the objects which have no
proper quotients. This technique has been employed very successfully
in graph theory, in particular in the study of distance transitive
graphs \cite{smith}, $s$-arc-transitive graphs \cite{praeger} and
locally $s$-arc-transitive graphs \cite{GLP1}. The success of such
methods relies on taking appropriate quotients so that the quotient
inherits the desired properties of the original object. Bipartite graphs are
incidence geometries of rank 2 for which the 
cited results apply. For higher ranks we don't have such insight. The aim of
this paper is to study quotients of incidence geometries in general and to
understand when a geometric property is inherited by a quotient. The
reverse operation 
of expanding a given geometry to discover families of geometries with
the given one as quotient is also important. Quotients of geometries
are featured in particular in the seminal 1981 paper of Tits
\cite{titslocal}, and in Pasini's standard reference \cite{Pasi94} for
diagram geometries. In these works various restrictions are made, both
on the geometries and the projection maps to the quotients. 

Like Tits, we consider quotients of what we now call pregeometries
(see below) and what Tits called `geometries'. Also, like Pasini (but
unlike Tits), the most general quotients we consider are constructed
modulo a type-refining partition (see Subsection \ref{sec:prelim}). We
also consider the special quotients that Tits considered (and that we
call `orbit-quotients'), constructed modulo the orbits of some
subgroup of automorphisms, and an even more special class called
`normal quotients' that proved effective in the case of
$s$-arc-transitive and locally $s$-arc-transitive graphs mentioned
above.  Fundamental problems that arise are: 

\begin{problem}
\label{prob:1}
Determine when the quotient of a geometry is a geometry.
\end{problem}

\begin{problem}
\label{prob:2}
Determine when the quotient of a flag-transitive geometry is a
flag-transitive geometry. 
\end{problem}

Because of the restricted type of quotients considered by Tits or
Pasini, these problems essentially did not arise in their
work. However, in a general investigation of the quotients of
geometries and pregeometries, they are among the most important
problems to consider. In Section~\ref{sec:survey}, we give a brief
summary of the kinds of quotients considered by Tits and Pasini,
together with some instructive examples to facilitate an understanding
of our results, and enable their comparison with earlier work. 

A partial answer to Problem \ref{prob:1} is given by the following theorem which is a combination of Lemmas \ref{lem:rank3}, \ref{lem:flagliftgeom}, \ref{lem:covgivesgeom}, \ref{lem:iscovering} and \ref{lem:shadowable}. The basic definitions for pregeometries are given in Subsection~\ref{sec:prelim},  covers are discussed in Section~\ref{sec:covers} and shadowable geometries are discussed in Section~\ref{sec:shadowable}. The \textsc{(FlagsLift)} condition is introduced in Subsection~\ref{sec:flagslift} and essentially states that every flag in the quotient arises as the projection of a flag in the original geometry.

\begin{theorem}
\label{thm:mainthm} 
Let $\Gamma$ be a geometry with type-refining partition $\mathcal{B}$. Then $\Gamma_{/\mathcal{B}}$ is a geometry if at least one of the following holds:
\begin{enumerate}
 \item $\Gamma$ has rank at most $3$,
 \item the condition \textsc{(FlagsLift)} holds for $\Gamma$,
\item $\Gamma$ is a cover of $\Gamma_{/\mathcal{B}}$, 
\item the projection $\pi_{/\mathcal{B}}$ is surjective on corank $1$ residues and $d(\alpha,\beta)\geq 4$ for distinct $\alpha,\beta$ in the same block,
 \item $\Gamma$ is shadowable and $\Gamma_{/\mathcal{B}}$ is an orbit-quotient.
\end{enumerate}
\end{theorem}

Example \ref{eg:notgeom} gives an example of a quotient $\Gamma_{/\mathcal{B}}$ of a rank 4 geometry $\Gamma$  for which \textsc{(FlagsLift)} does not hold, $\Gamma$ is not a cover of $\Gamma_{/\mathcal{B}}$ and $\Gamma_{/\mathcal{B}}$ is not a geometry (so the conditions in parts 1, 2 and 3 cannot be relaxed). Also Example \ref{D4example} gives a natural infinite family of rank 4 geometries 
arising from orthogonal geometry where the ({\sc FlagsLift}) condition 
does not hold for a certain orbit quotient. We see in Remark \ref{rem:coseteg} that even a normal quotient of a flag-transitive geometry need not be a geometry and even if it is a geometry it need not satisfy the \textsc{(FlagsLift)} condition. In
Constructions \ref{con:blowup} and \ref{con:liftshadowable} we give two different ways of lifting a
geometry to a larger geometry which has the initial geometry as a quotient.

In Section \ref{sec:props} we explore the geometric properties of connectivity and firmness to determine the impact of the quotienting operation. Section \ref{sec:diagram} uses the diagram of a geometry to deduce information about a quotient.

If a group $G$ acts as automorphisms of a pregeometry $\Gamma$ and preserves a type-refining partition then $G$ also induces automorphisms of the corresponding quotient. We have the following necessary and sufficient condition for Problem \ref{prob:2}.

\begin{theorem}
\label{thm:flagtrans}
 Suppose that $G$ is flag-transitive on a geometry $\Gamma$ and that $\mathcal{B}$ is a type-refining partition invariant under $G$. Then $G$ is flag-transitive on $\Gamma_{/\mathcal{B}}$ if and only if \textsc{(FlagsLift)} holds.
\end{theorem}

The proof of  Theorem \ref{thm:flagtrans} can be found in Section~\ref{sec:gractions}.

Despite the fact that the class of geometries is not closed under quotients, if we widen our attention to the class of coset pregeometries (see Subsection \ref{sec:coset}) we obtain a class of pregeometies which is closed under quotients and contains all flag-transitive geometries.

\begin{theorem}
\label{thm:cosetquot}
Let $\Gamma$ be a coset pregeometry for a group $G$ and let $\mathcal{B}$ be a type-refining partition invariant under $G$. Then $\Gamma_{/\mathcal{B}}$ is a coset pregeometry.
\end{theorem}

Our results do not solve Problems \ref{prob:1} and \ref{prob:2} completely and several open questions are posed in the text. 

\section*{Acknowledgements}
The authors thank anonymous referees for advice that has led to a 
clearer exposition. Moreover they thank a referee of an earlier version 
  for pointing out several infelicities and providing several 
illuminating examples included in the text.

\section{Quotients and lifting flags}

\subsection{Definitions and preliminaries}
\label{sec:prelim}
A \emph{pregeometry} $\Gamma=(X,*,t)$ is a set $X$, whose members are
called \emph{elements}, with a symmetric reflexive relation $*$ and a
map $t$ from $X$ onto some set $I$ whose elements are called
\emph{types}. Moreover, $\alpha*\beta$ and $t(\alpha)=t(\beta)$
implies $\alpha=\beta$. The relation $*$ is the incidence relation and
if $\alpha*\beta$ we say that $\alpha$ and $\beta$ are
\emph{incident}. The \emph{rank} of $\Gamma$ is $|I|$. \emph{We will
  often just refer to a pregeometry $\Gamma$, in which case we take it
  to represent the triple $(X,*,t)$}, and similarly a pregeometry
$\Gamma'$ will mean the triple $(X',*',t')$. For each $i\in I$, we let
$X_i=t^{-1}(i)$, so that $X=\dotcup_{i\in I} X_i$. This partition of
$X$ is known as the \emph{type partition}.  

A \emph{flag} of $\Gamma$ is a set of pairwise incident elements. By
definition a flag contains at most one element of each type.  The
\emph{rank} of a flag $F$ is $|t(F)|$ while its \emph{corank} is
$|I|-|t(F)|$.  
We refer to $t(F)$ as the type of $F$ and $I\setminus t(F)$ as its
\emph{cotype}. A \emph{chamber} is a flag of rank $|I|$. We call
$\Gamma$ a \emph{geometry} if all maximal flags are chambers. We say
that a geometry is \emph{firm} if every corank 1 flag is contained in
at least two chambers. 

An \emph{automorphism} of $\Gamma$ is a permutation of $X$ which
preserves incidence and which fixes each $X_i=t^{-1}(i)$ setwise. The
set of all automorphisms is denoted $\Aut(\Gamma)$. 

The graph $(X,E)$, where $E$ is the set of all rank $2$ flags of
$\Gamma$, is called the \emph{incidence graph} of $\Gamma=(X,*,t)$. We
say that $\Gamma$ is \emph{connected} if its incidence graph is
connected.  

Let $\Gamma=(X,*,t)$ be a pregeometry and let $\mathcal{B}$ be a
partition of $X$ which is a refinement of the type partition of
$\Gamma$. We call such a partition a \emph{type-refining
  partition}. We define a new pregeometry
$\Gamma_{/\mathcal{B}}=(\mathcal{B},*_{/\mathcal{B}},t_{/\mathcal{B}})$
where  
\begin{enumerate}
\item $B_1 *_{/\mathcal{B}} B_2$ if and only if there exist $\alpha_i\in
  B_i$ (for $i=1,2$) with $\alpha_1*\alpha_2$, 
\item $t_{/\mathcal{B}}:\mathcal{B}\rightarrow I$ and
  $t_{/\mathcal{B}}(B)=t(\alpha)$ for $\alpha\in B$. 
\end{enumerate}
For every quotient $\Gamma_{/\mathcal{B}}$ in this paper, the
associated partition $\mathcal{B}$ will be type-refining. The quotient
$\Gamma_{/\mathcal{B}}$ yields a projection  
$\pi_{/\mathcal{B}}:\Gamma\rightarrow \Gamma_{/\mathcal{B}}$ (namely
$\pi_{/\mathcal{B}}(\alpha)$ is the part of $\mathcal{B}$ containing
$\alpha$) which is a pregeometry morphism from $\Gamma$ onto
$\Gamma_{/\mathcal{B}}$, that is a map from $X$ onto $\mathcal{B}$
which preserves incidence and type. However, nonincidence is not in general preserved.

If the partition $\mathcal{B}$ of $X$ is the set of orbits of some
subgroup $A$ of $\Aut(\Gamma)$ we denote $\Gamma_{/\mathcal{B}}$ by
$\Gamma_{/A}$ and call it an \emph{orbit-quotient}. Moreover, when
$A\norml \Aut(\Gamma)$, we call $\Gamma_{/A}$ a \emph{normal
  quotient}.  

Given a flag $F$ of a pregeometry $\Gamma$, the \emph{residue} of $F$
in $\Gamma$, denoted $\Gamma_F$, is the pregeometry $(X_F,*_F,t_F)$
induced by $\Gamma$ on the set of elements $X_F$ incident with every
member of $F$ and whose type is not in $t(F)$.  If $\Gamma$ is a
geometry then so is $\Gamma_F$. In what follows, a corank $1$ residue will be used as a synonym for the residue of a flag of rank $1$.

For a pregeometry $\Gamma$ with quotient $\Gamma_{/\mathcal{B}}$, the projection map $\pi_{/\mathcal{B}}$ induces a pregeometry homomorphism from the corank $1$ residue $\Gamma_{\alpha}$ to $(\Gamma_{/\mathcal{B}})_{\pi_{/\mathcal{B}}(\alpha)}$. This homomorphism may not be onto as it is easy to construct examples where there are elements of $\pi_{/\mathcal{B}}(\alpha)$ that are incident to elements contained in blocks that do not contain an element incident with $\alpha$. However, for orbit-quotients we have that $\pi_{/\mathcal{B}}$ is surjective on corank $1$ residues.
\begin{lemma}
 \label{lem:orbitressurj}
Let $\Gamma$ be a pregeometry with orbit-quotient $\Gamma_{/A}$. Then for each element $\alpha$ of $\Gamma$, $\pi_{/A}(\Gamma_{\alpha})=(\Gamma_{/A})_{\alpha^A}$.
\end{lemma}
\begin{proof}
Let $B\in(\Gamma_{/A})_{\alpha^A}$. Then there exist $\beta\in B$ and $\alpha'\in \alpha^A$ such that $\beta*\alpha'$. Also there exists $a\in A$ such that $(\alpha')^a=\alpha$ and hence $\beta^a\in B\cap\Gamma_{\alpha}$. Thus $\pi_{/A}(\beta^a)=B$ and so $\pi_{/A}(\Gamma_{\alpha})=(\Gamma_{/A})_{\alpha^A}$.
\end{proof}

\subsection{Lifting flags}
\label{sec:flagslift}

Clearly given a flag $F$ of $\Gamma$, $\pi_{/\mathcal{B}}(F)$ is a
flag of $\Gamma_{/\mathcal{B}}$. However, the converse is not true in
general, that is, given a flag $F_{\mathcal{B}}$ in
$\Gamma_{/\mathcal{B}}$ there may not be a flag in $\Gamma$ which
projects onto $F_{\mathcal{B}}$. This is illustrated in the example
below. 

\begin{example}
\label{eg:6cyc}
Let $\Gamma$ be the rank $3$ geometry whose incidence graph is given
on the left of Figure \ref{fig:6cyc} such that the type partition is
given by the ellipses. Let $\mathcal{B}$ be the partition given by the
rectangular boxes. Then $\Gamma_{/\mathcal{B}}$ is a geometry and its
incidence graph is given on the right, again with the type partition
given by the ellipses. The rank 3 flag $\{A,B,C\}$ in
$\Gamma_{/\mathcal{B}}$ does not arise from a flag in $\Gamma$. Note
that $\la (\alpha,\beta)(\gamma,\delta)(\epsilon,\lambda)\ra$ is a
group of automorphisms of $\Gamma$ with orbits the parts of
$\mathcal{B}$. Hence $\Gamma_{/\mathcal{B}}$ is in fact an
orbit-quotient. The full automorphism group of $\Gamma$ is isomorphic to $S_3$.
\end{example}

\psfrag{A}{$A$}\psfrag{B}{$B$}\psfrag{C}{$C$}
\psfrag{alpha}{$\alpha$}\psfrag{beta}{$\beta$}
\psfrag{gamma}{$\gamma$}\psfrag{delta}{$\delta$}
\psfrag{epsilon}{$\varepsilon$}\psfrag{lambda}{$\lambda$}

\begin{figure}
\begin{center}
\includegraphics[width=0.45\textwidth]{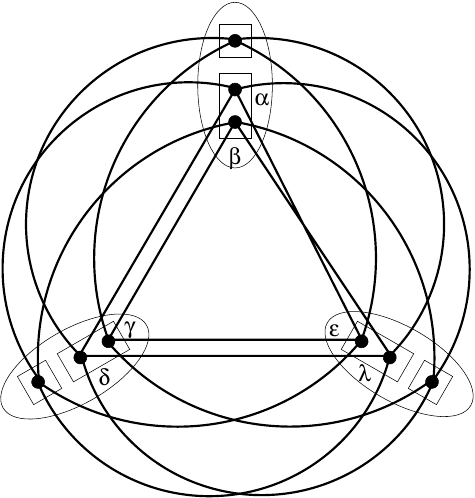}\quad\quad
\includegraphics[height=6cm,width=0.45\textwidth]{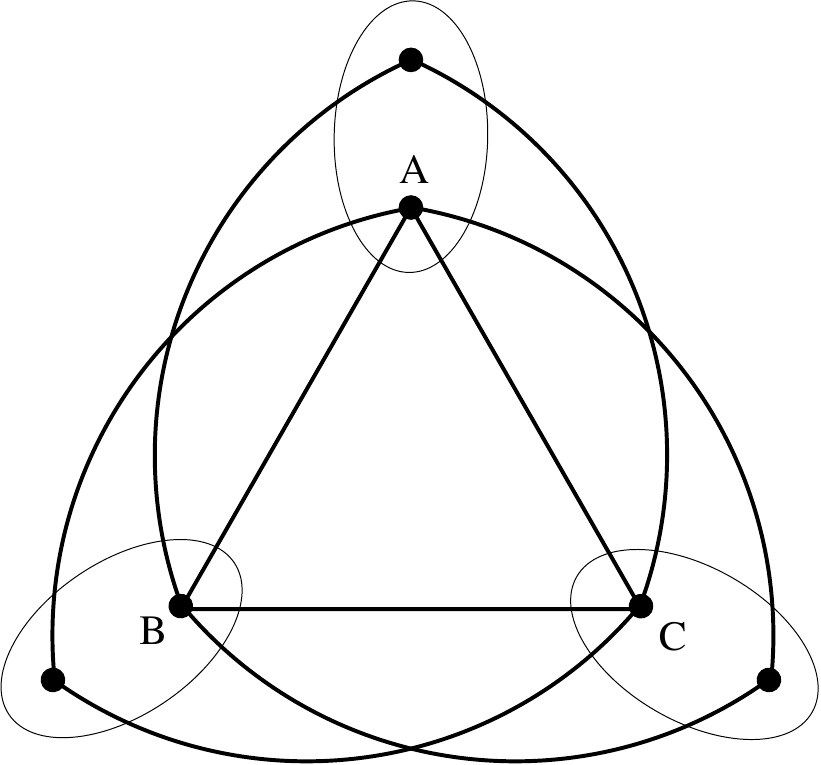}
\caption{A rank 3 geometry (on the left) whose quotient (on the
  right) has chambers which do not lift.} 
\label{fig:6cyc}
\end{center}
\end{figure}

This illustrates an issue at the heart of Problem \ref{prob:1}. The
failure of flags in a quotient to lift naturally to flags in the
original geometry may prevent a quotient of a geometry being a
geometry, since such a flag in a quotient may not be contained in a
chamber. To overcome these difficultes we introduce the following condition

\begin{quote}
\textsc{(FlagsLift)} For each flag $F_{\mathcal{B}}$ of
$\Gamma_{/\mathcal{B}}$ there exists a flag $F$ of $\Gamma$ such that
\mbox{$\pi_{/\mathcal{B}}(F)=F_{\mathcal{B}}$.} 
\end{quote}
This condition is weaker than the conditions of Tits and Pasini
discussed further in Section~\ref{sec:survey}. 

The problem of flags failing to lift does not prevent the quotient being a geometry for the example in Figure \ref{fig:6cyc}, as the only flags that do not lift are chambers. In
fact quotients of geometries of rank at most 3 are always
geometries. 

\begin{lemma}
\label{lem:rank3}
Let $\Gamma$ be a geometry of rank at most 3 and let $\mathcal{B}$
be a type-refining partition. Then $\Gamma_{/\mathcal{B}}$ is a
geometry. 
\end{lemma}
\begin{proof}
If $\Gamma$ has rank 1 then $\Gamma_{/\mathcal{B}}$ is trivially a
geometry. So suppose $\Gamma$ has rank at least 2, and let
$B\in\mathcal{B}$ and $\alpha\in B$. Since $\Gamma$ is a geometry,
there exists $\beta\in\Gamma$ with $\alpha*\beta$. Let $C$ be the
unique element of $\mathcal{B}$ containing $\beta$. Then $\{B,C\}$ is
a flag of $\Gamma_{/\mathcal{B}}$. If $\Gamma$ has rank 2 it follows
that $\Gamma_{/\mathcal{B}}$ is a geometry. Suppose now that $\Gamma$
has rank 3 and suppose that $\{B,C\}$ is an arbitrary rank 2 flag
of $\Gamma_{/\mathcal{B}}$. By definition, there exist $\alpha\in B$
and $\beta\in C$ such that $\alpha*\beta$. Since $\Gamma$ is a
geometry, there exists $\gamma$ incident with $\alpha$ and
$\beta$. Then if $D\in\mathcal{B}$ is the unique subset containing
$\gamma$, it follows that $\{B,C,D\}$ is a chamber. Hence any flag of
$\Gamma_{/\mathcal{B}}$ is contained in a chamber and so
$\Gamma_{/\mathcal{B}}$ is also a geometry in the rank 3 case. 
\end{proof}

For larger rank it is no longer the case in general that the quotient of a geometry is a geometry. 

\begin{example}
\label{eg:notgeom}
Let $\Gamma$ be the rank 4 geometry with incidence graph given on
the left of Figure \ref{fig:rank4geom} whose parts of the type
partition are the three ellipses plus the set consisting of the three
remaining points. Let $\mathcal{B}$ be the type-refining partition
whose blocks are given by the rectangular boxes and each of the three
remaining points.  Note that $\la
(\alpha,\beta)(\gamma,\delta)(\epsilon,\lambda)\ra$ is a group of
automorphisms of $\Gamma$ whose orbits are the parts of
$\mathcal{B}$. Hence $\Gamma_{/\mathcal{B}}$ is an orbit-quotient. 
The incidence graph of $\Gamma_{/\mathcal{B}}$ is given on the right.
From this we can see that the flag $\{A,B,C\}$ is not contained in a
chamber and so $\Gamma_{/\mathcal{B}}$ is not a geometry. Note that
$\{A,B,C\}$ does not arise from a flag of $\Gamma$. 
\end{example}

\begin{figure}
\begin{center}
\includegraphics[width=0.45\textwidth,height=8cm]{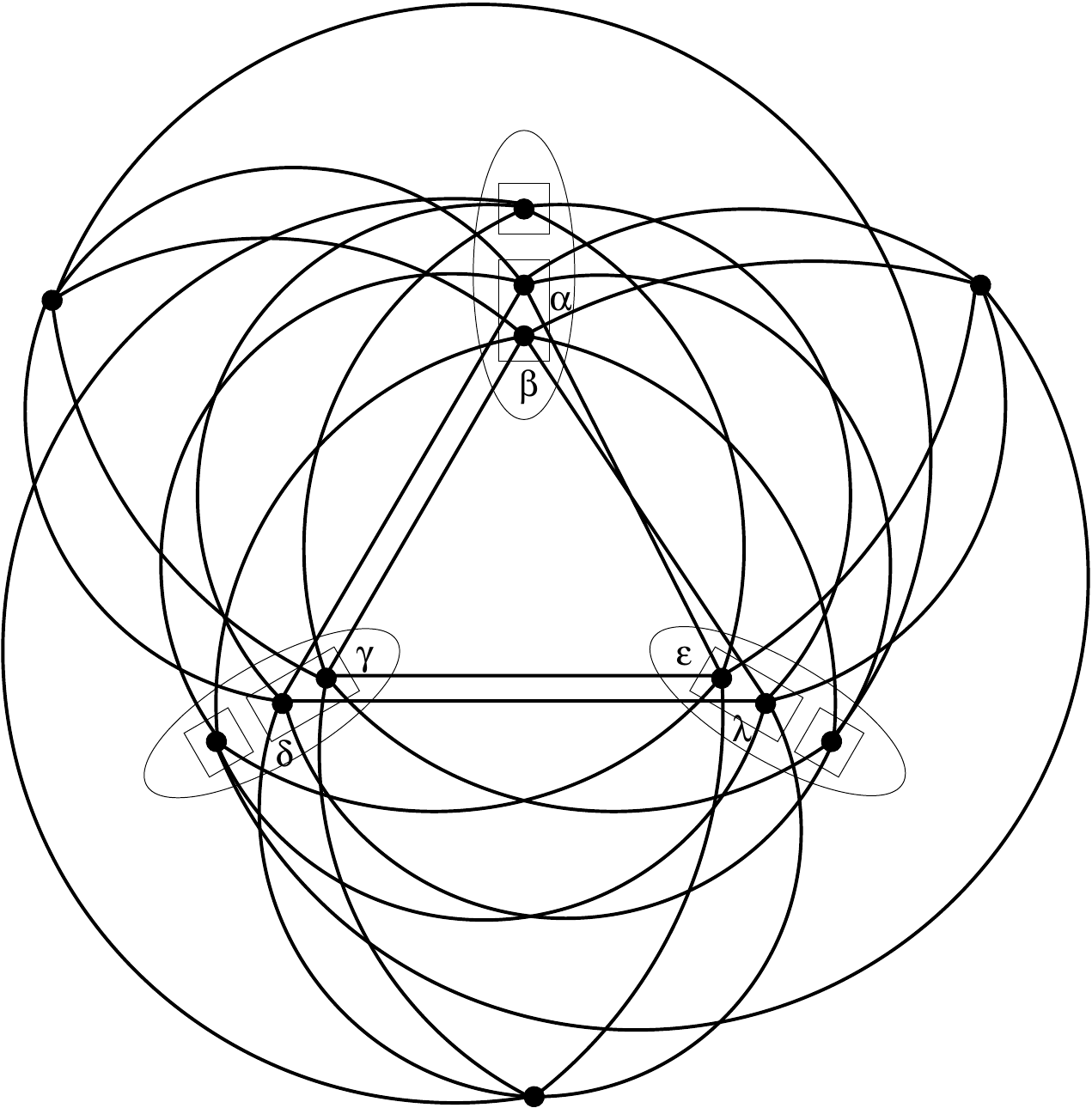}\quad\quad
\includegraphics[width=0.45\textwidth,height=8cm]{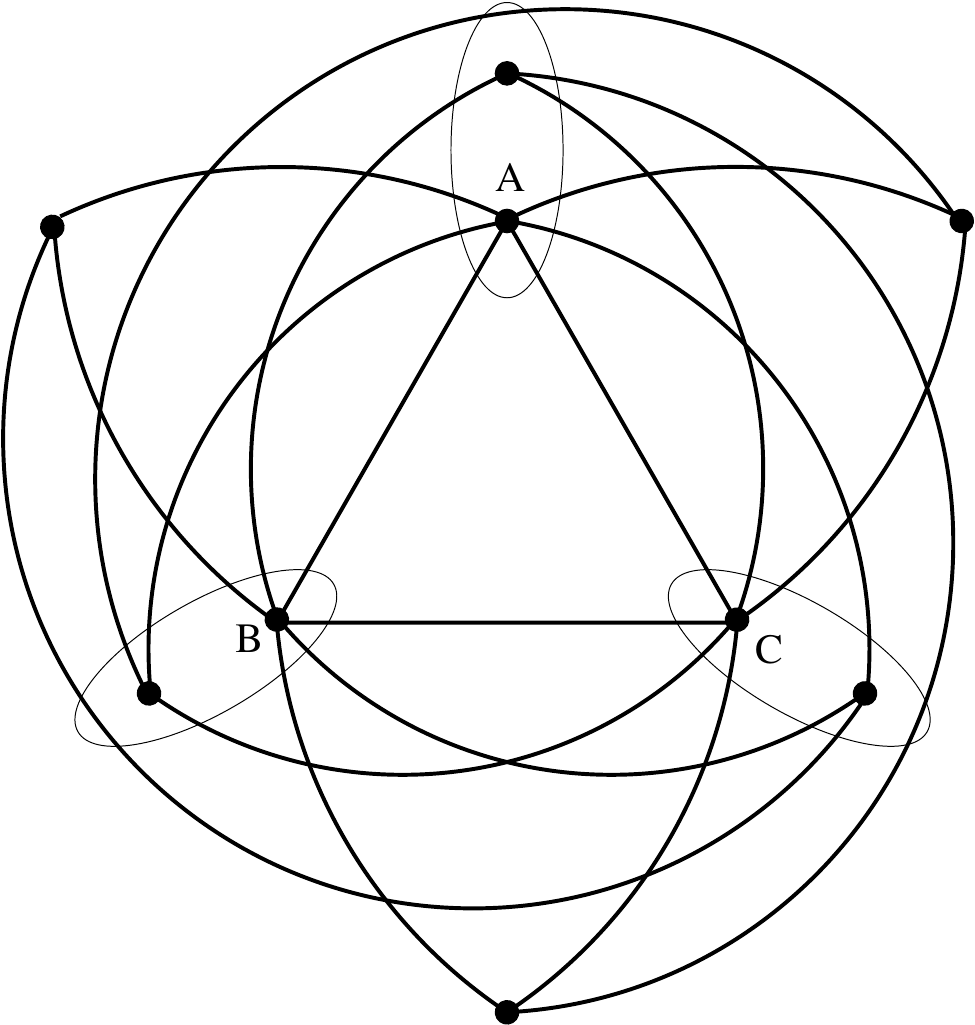}
\caption{A rank 4 geometry (on the left) whose quotient (on the
  right) is not a geometry} 
\label{fig:rank4geom}
\end{center}
\end{figure}

The condition \textsc{(FlagsLift)} is sufficient to guarantee that the quotient of a geometry is a
geometry. 

\begin{lemma}
\label{lem:flagliftgeom}
Let $\Gamma$ be a geometry and let $\mathcal{B}$ be a type-refining
partition. If \textsc{(FlagsLift)} holds then $\Gamma_{/\mathcal{B}}$
is a geometry. 
\end{lemma}
\begin{proof}
Let $F_{\mathcal{B}}$ be a flag in  $\Gamma_{/\mathcal{B}}$. Then
there exists a flag $F$ of $\Gamma$ which projects onto
$F_{\mathcal{B}}$. Since $\Gamma$ is a geometry, there exists a
chamber $F'$ of $\Gamma$ containing $F$. Then $\pi_{/\mathcal{B}}(F')$
is a chamber of $\Gamma_{/\mathcal{B}}$ containing $F_{\mathcal{B}}$
and so $\Gamma_{/\mathcal{B}}$ is a geometry. 
\end{proof}

\section{Covers}
\label{sec:covers}

For two pregeometries $\Gamma$, $\Gamma'$ with the same type set $I$,
we say that a morphism (that is, an incidence and type preserving
map)  $h:\Gamma\rightarrow \Gamma'$ is an \emph{$m$-covering} if for
each flag $F$ of $\Gamma$ of corank $m$, the restriction of $h$ to the
residue $\Gamma_F$ is an isomorphism onto the residue
$\Gamma'_{h(F)}$. In this case we say that $\Gamma$ is an
\emph{$m$-cover} of $\Gamma'$. An $m$-cover is also a $k$-cover whenever
$0<k<m$. When $m=|I|-1$, that is for each
element $\alpha$ of $\Gamma$, the restriction of $h$ to the
residue $\Gamma_\alpha$ is an isomorphism onto the residue
$\Gamma'_{h(\alpha)}$, we say that $h$ is a \emph{covering} and $\Gamma$ is a
\emph{cover} of $\Gamma'$. 

Properties such as \textsc{(FlagsLift)} hold for covers and in
particular, being a cover is sufficient for a positive solution to
Problems \ref{prob:1} and \ref{prob:2} (see also Theorem
\ref{thm:flagtrans}). 

\begin{lemma}
\label{lem:coveringflags}
Suppose that, for a pregeometry $\Gamma$, $\pi_{/\mathcal{B}}$ is a
covering and let $F_{\mathcal{B}}$ be a flag of
$\Gamma_{/\mathcal{B}}$. Then \textsc{(FlagsLift)} holds, that is,
there exists a flag $F$ of $\Gamma$ such that
$F_{\mathcal{B}}=\pi_{/\mathcal{B}}(F)$. 
\end{lemma}
\begin{proof}
If $F_{\mathcal{B}}$ has rank 1 then $F_{\mathcal{B}}=\{B\}$ for
some $B\in\mathcal{B}$ and $\alpha\in B$ yields the required flag of
$\Gamma$.  Thus suppose that the rank of $F_{\mathcal{B}}$ is at least
2. Let $B\in F_{\mathcal{B}}$. Then $F_{\mathcal{B}}\setminus \{B\}$
is a flag of rank at least 1 in the residue
$(\Gamma_{/\mathcal{B}})_B$. Let $\alpha\in B$. Then as
$\pi_{/\mathcal{B}}$ is a covering it induces an isomorphism from
$\Gamma_\alpha$ onto $(\Gamma_{/\mathcal{B}})_B$. Hence there exists a
flag $F'$ in $\Gamma_{\alpha}$ such that
$\pi_{/\mathcal{B}}(F')=F_{\mathcal{B}}\backslash \{B\}$. Moreover,
$F'\cup\{\alpha\}$ is also a flag of $\Gamma$ and
$\pi_{/\mathcal{B}}(F'\cup\{\alpha\})=F_{\mathcal{B}}$. Thus
$F=F'\cup\{\alpha\}$ has the required property. 
\end{proof}

\begin{lemma}
\label{lem:covgivesgeom}
If $\Gamma$ is a geometry of rank at least~2 and $\pi_{/\mathcal{B}}$
is a covering, then 
$\Gamma_{/\mathcal{B}}$ is a geometry. Moreover, if $\Gamma$ is firm
then so is $\Gamma_{/\mathcal{B}}$.  
\end{lemma}
\begin{proof}
By Lemma \ref{lem:coveringflags} and Lemma \ref{lem:flagliftgeom},
$\Gamma_{/\mathcal{B}}$ is a geometry. Suppose now that $\Gamma$ is
firm and let $F_{\mathcal{B}}$ be a flag of corank~1 in
$\Gamma_{/\mathcal{B}}$. Then by Lemma \ref{lem:coveringflags}, there
exists a flag $F$ of $\Gamma$ such that
$\pi_{/\mathcal{B}}(F)=F_{\mathcal{B}}$. Since $\Gamma$ is firm, $F$
is contained in two distinct chambers $F',F''$. Now $\pi_{/\mathcal{B}}(F')$
and $\pi_{/\mathcal{B}}(F'')$ are chambers of $\Gamma_{/\mathcal{B}}$ containing
$F_{\mathcal{B}}$. Since $\pi_{/\mathcal{B}}$ is a covering, it
induces an isomorphism from the residue $\Gamma_F$ onto
$(\Gamma_{/\mathcal{B}})_{F_{\mathcal{B}}}$.  Since $F'\setminus F$
and $F''\setminus F$ are distinct elements of $\Gamma_F$ it follows that
$\pi_{/\mathcal{B}}(F')\neq \pi_{/\mathcal{B}}(F'')$ and so
$F_{\mathcal{B}}$ is also contained in two chambers. Hence
$\Gamma_{/\mathcal{B}}$ is firm. 
\end{proof}

In the case of orbit-quotients, the property of being a cover
restricts the action of the group. 
\begin{lemma}
\label{lem:semireg}
Let $\Gamma$ be a connected pregeometry and let
$A\leqslant\Aut(\Gamma)$. If $\Gamma$ is a cover of the orbit-quotient
$\Gamma_{/A}$ then $A$ acts semiregularly on $X$. 
\end{lemma}
\begin{proof}
Let $\alpha\in X$ and let $B=\alpha^A$.  Let  $g\in A_\alpha$. Then
$g$ fixes $\Gamma_\alpha$ setwise. Since $\Gamma$ is a cover of
$\Gamma_{/A}$, $\Gamma_\alpha\cong (\Gamma_{/A})_B$ and hence for each
$C*_{/\mathcal{B}}B$, there is a unique element $\gamma\in C$ such
that $\gamma*\alpha$. Since $g$ fixes $\alpha$ and $C$ setwise, $\gamma^g\in C$ and by uniqueness $\gamma^g=\gamma$. So $g$ fixes $\Gamma_\alpha$ pointwise.  Let
$\beta\in X$. Since $\Gamma$ is connected there exist elements
$\alpha=\alpha_0,\alpha_1,\ldots,\alpha_k=\beta$  such that
$\alpha_0*\alpha_1*\cdots*\alpha_k$. Now $\alpha_1\in\Gamma_\alpha$
and so $g\in A_{\alpha_1}$. Thus $g$ fixes $\Gamma_{\alpha_1}$
pointwise and hence fixes $\alpha_2$. It follows inductively that $g$
fixes $\alpha_k=\beta$ and so $g$ fixes each element of $X$. Thus
$g=1$ and the result follows. 
\end{proof}

\begin{lemma}
Suppose $\pi_{/\mathcal{B}}$ is a covering and $\alpha\neq\beta$ are
in the same block $B$ of $\mathcal{B}$. Then $d(\alpha ,\beta)\geq 3$.
\end{lemma}
\begin{proof}
Since $\mathcal{B}$ is type-refining, $\alpha$ and $\beta$ must have
the same type and by the definition of a pregeometry, $d(\alpha
,\beta)>1$ so that we only have to show that $d(\alpha ,\beta)\neq
2$. If some $\gamma$ has the property $\alpha\ast\gamma\ast\beta$, we
get that both $\alpha$ and $\beta$ belong to the residue
$\Gamma_{\gamma}$, on which $\pi_{/\mathcal{B}}$ is supposed to be
injective. This is a contradiction since both $\alpha$ and $\beta$ map
to the block $B$.
\end{proof}

We note that $d(\alpha,\beta)\geq 3$ does not imply a cover. For example, let $\Gamma$ be the pregeometry whose elements are the vertices of a hexagon, incidence is adjacency in the cycle and two elements are of the same type if they are antipodal. If we let $\mathcal{B}$ be the partition of the elements into those of the same type then $d(\alpha,\beta)=3$ for all $\alpha,\beta$ in the same block. However, $\Gamma_{/\mathcal{B}}$ has incidence graph a cycle of length three and so $\Gamma$ is not a cover of $\Gamma_{/\mathcal{B}}$.

\begin{lemma}\label{lem:3inj}
If for each $B\in\mathcal{B}$ and distinct $\alpha,\beta\in B$ we have
$d(\alpha,\beta)\geq 3$ in the incidence graph, then the restriction
of the projection $\pi_{/\mathcal{B}}$ to any corank~$1$ residue is
injective.
\end{lemma}
\begin{proof}
Let $\alpha$ be an element of $\Gamma$ and $B$ be the unique element
of $\mathcal{B}$ containing $\alpha$. Since all elements of
$\Gamma_\alpha$ are incident with $\alpha$, for each
$\beta\in\Gamma_\alpha$, the unique element $D\in \mathcal{B}$
containing $\beta$ meets $\Gamma_{\alpha}$ only in $\beta$. Hence the
restriction of $\pi_{/\mathcal{B}}$ to $\Gamma_\alpha$ is injective. 
\end{proof}
 
When $\pi_{/\mathcal{B}}$ is surjective on corank~$1$ residues, for example, when $\mathcal{B}$ is the set of orbits of some subgroup (see Lemma \ref{lem:orbitressurj}), we have the following lemma.

\begin{lemma}\label{lem:iscovering}
Assume that $\pi_{/\mathcal{B}}$ is surjective on corank~$1$ residues.
If the distance between any two distinct elements of the same
block of $\mathcal{B}$ is at least $4$, then $\pi_{/\mathcal{B}}$ is a
covering. 
\end{lemma}
\begin{proof}
By Lemma~\ref{lem:3inj} we have that the restriction of
$\pi_{/\mathcal{B}}$ to a corank~1 residue is a bijection. We now show
that this bijection also preserves nonincidence. Suppose $\beta$ and $\gamma$
are nonincident elements in some residue $\Gamma_\alpha$, with $\beta$, $\gamma$ 
contained in the blocks $B$ and $C$ respectively. If
$B\ast_{/\mathcal{B}} C$, the residue $(\Gamma_{/\mathcal{B}})_B$ must
contain the block $C$. Since the restriction of $\pi_{/\mathcal{B}}$
to $\Gamma_\beta$ is a bijection onto $(\Gamma_{/\mathcal{B}})_B$, there
must be a unique element $\gamma'\in C$ that belongs to $\Gamma_\beta$. There
are two possibilities: either $\gamma'=\gamma$ or $\gamma'\neq \gamma$. In the former case
we get $\gamma\ast \beta$, a contradiction. The latter case yields a path
$\gamma'\ast \beta\ast \alpha\ast \gamma$ of length~3 joining $\gamma'$ to $\gamma$ which both lie
in $C$, another contradiction.
\end{proof}

This distance condition was considered by Tits (see the discussion in Section \ref{sec:survey}). 
The following example is a good illustration of coverings and quotients.

\begin{example}
Let $\Gamma=(X,*,t)$ be the rank 3 geometry formed from a
triangulation of the Euclidean plane with triangles coloured black or
white such that no two triangles sharing a side have the same
colour. Let $X_1$ be the set of vertices of the triangles, $X_2$ the
set of white triangles and $X_3$ the set of black triangles. Elements
of $X_1$ are incident with the triangles they are contained in, while
a black triangle is incident with a white triangle  if and only if they
share an edge. Each flag is contained in a chamber so $\Gamma$ is a
geometry. Let $G=L\rtimes D_6$ where $L$ is the group of all
translations by elements of the lattice of vertices. Then $L$ acts
regularly on $X_1, X_2$ and $X_3$. Let $n\geq 2$ be an integer and let
$N=\{n\ell\mid \ell\in L\}\norml L$. Then $N$ is intransitive on $X_i$
for $i=1,2,3$ with $n^2$ orbits on each. Moreover, for each $\alpha\in
X$ and $n\ell\in N\setminus\{1\}$, the distance in the incidence graph of $\Gamma$ between $\alpha$ and
$\alpha^{n\ell}$ is at least $2n\geq 4$. Hence by Lemma \ref{lem:3inj} and Lemma
\ref{lem:iscovering}, the projection
$\pi_N:\Gamma\rightarrow\Gamma_{/N}$ is a covering and so by Lemma
\ref{lem:covgivesgeom},  $\Gamma_{/N}$ is a geometry. 
\end{example}

Let $\Sigma$ be a graph and let $\mathcal{B}$ be a partition of the
vertex set of $\Sigma$. The quotient graph $\Sigma_{\mathcal{B}}$ of
$\Sigma$ with respect to $\mathcal{B}$ is the graph with vertices the
blocks of $\mathcal{B}$ and two blocks $B_1,B_2\in\mathcal{B}$ are
adjacent if there exists $\alpha\in B_1$, $\beta\in B_2$ such that
$\alpha$ and $\beta$ are adjacent in $\Sigma$. If $\alpha$ is adjacent to $\beta$ we say that $\alpha$ is a \emph{neighbour} of $\beta$. We say that $\Sigma$ is
a \emph{cover} of $\Sigma_{\mathcal{B}}$ if for all vertices $\sigma$ of
$\Sigma$ with unique block $B\in\mathcal{B}$ containing $\sigma$, the
projection map $\pi:\Sigma\rightarrow\Sigma_{\mathcal{B}}$ induces a
bijection between the set of neighbours of
$\sigma$ in $\Sigma$ and the set of neighbours of
$B$ in $\Sigma_{\mathcal{B}}$. This is equivalent to requiring that
the induced subgraph between any two adjacent blocks of $\mathcal{B}$
is a complete matching. Much attention has been paid to constructing covers of graphs in the
literature, see for example \cite{biggs,djok,grosstucker,MNS}. 

The definition of a covering of a graph is weaker than that of a geometry: a covering of a geometry requires an \emph{isomorphism} between the graph induced on the set of neighbours of
$\sigma$ in the incidence graph of a geometry $\Gamma$ and the graph induced on the set of neighbours of
$B$ in the incidence graph of $\Gamma_{/\mathcal{B}}$, instead of merely a bijection. However, in the orbit-quotient case, the ability to lift rank~3 flags of $\Gamma_{/\mathcal{B}}$ to rank~3 flags of $\Gamma$ uniquely defines the difference between the two types of covers.

\begin{lemma}
Let $\Gamma$ be a geometry with orbit-quotient $\Gamma_{/A}$.   Then $\Gamma$ is a cover of $\Gamma_{/A}$ if and only if the following hold:
\begin{enumerate}
 \item the incidence graph of $\Gamma$ is a graph cover of the incidence graph of $\Gamma_{/A}$, and
 \item for each rank~3 flag $\{B,C,D\}$ of $\Gamma_{/A}$, there exists $\beta\in B,\gamma\in C$ and $\delta\in D$ such that $\{\beta,\gamma,\delta\}$ is a flag in $\Gamma$.
\end{enumerate}
\end{lemma} 
\begin{proof}
Suppose first that $\Gamma$ is a cover of $\Gamma_{/A}$ and let $\{B,C,D\}$ be a flag of $\Gamma_{/A}$. Let $\beta\in B$. Then $\pi_{/A}$ induces an isomorphism from $\Gamma_{\beta}$ to $(\Gamma_{/A})_B$ and so the incidence graph of $\Gamma$ is a cover of the incidence graph of $\Gamma_{/A}$. Moreover, there exist $\gamma\in C\cap\Gamma_{\beta}$ and $\delta\in D\cap\Gamma_{\delta}$ such that $\gamma*\delta$. Hence $\{\beta,\gamma\,\delta\}$ is a flag in $\Gamma$.

Conversely, suppose that the incidence graph of $\Gamma$ is a cover of the incidence graph of $\Gamma_{/A}$ and we can lift rank 3 flags of $\Gamma_{/A}$ to rank 3 flags of $\Gamma$. Let $\beta$ be an element of $\Gamma$ and $B =\beta^A$. Since the incidence graph of $\Gamma$ is a cover of the incidence graph of $\Gamma_{/A}$, the projection map $\pi_{/A}$ induces a bijection from $\Gamma_{\beta}$ to $(\Gamma_{/A})_B$. It remains to show that $\pi_{/A}$ induces an isomorphism.

Now, if $\gamma,\delta\in \Gamma_{\beta}$ such that $\gamma*\delta$ we have  $\pi_{/A}(\gamma)*_{/A}\pi_{/A}(\delta)$. Conversely, suppose that $C*_{/A} D$ in $(\Gamma_{/A})_B$. Then $\{B,C,D\}$ is a flag in $\Gamma_{/A}$ and so by the assumption there exist $\beta'\in B,\gamma'\in C$ and $\delta'\in D$ such that $\{\beta',\gamma',\delta'\}$ is a flag in $\Gamma$. Since $B=\beta^A$, there exists $a\in A$ such that $(\beta')^a=\beta$. Hence $\{\beta,(\gamma')^a, (\delta')^a\}$ is a flag of $\Gamma$ which projects onto $\{B,C,D\}$ and  since $\pi_{/A}$ induces a bijection from $\Gamma_\beta$ to $(\Gamma_{/A})_B$ it follows that $\pi_{/A}$ induces an isomorphism from $\Gamma_{\beta}$ to $(\Gamma_{/A})_B$.
\end{proof}

\subsection{Affine examples and generalisations}

Coverings are a very specific case where Problem \ref{prob:1} has
a positive solution. We have already seen that the 
quotient of any rank 3 geometry is a geometry. The following
example is another instance where the original geometry is not a cover
of the quotient but the quotient is a geometry. 

\begin{example}
\label{eg:affine}
Let $\Gamma$ be the affine space $\AG(d,q)$ for some prime power $q$ and
$d\geq 3$. Let $G=\AGL(d,q)$ and $N$ be the normal subgroup of $G$
consisting of all translations. Then $\Gamma=(X,*,t)$ with
$|I|=\{0,\ldots,d-1\}$ such that $X_i$ consists of all affine
$i$-spaces and $*$ is the natural incidence. Now $N$ is transitive on
$X_0$, while it has orbits of length
$\frac{q^{d-i}(q^d-1)\cdots(q^{d-i+1}-1)}{(q^i-1)\cdots(q-1)}$ on
$X_i$ for $1\leq i\leq d-1$. Moreover, if we exclude $X_0$ from
$\Gamma_{/N}$ then $\Gamma_{/N}$ is the projective space
$\PG(d-1,q)$. The singleton $\{X_0\}$ of $\Gamma_{/N}$ is incident with all elements of $\Gamma_{/N}$ and so $\Gamma_{/N}$ is a geometry.

Each orbit of $N$ on $X_i$ for $i\geq 1$ contains precisely one
$i$-dimensional subspace of $\GF(q)^d$, with the remaining elements of
the orbit being its translates. If $W+x\in X_i$ for $i\leq d-2$, then
the distance between distinct $W+x$ and $(W+x)^n=W+x+n$ in the
incidence graph of $\Gamma$ is two for all $n\in N\setminus W$ as both are
incident with the translate $\la W,n\ra+x$ of the $(i+1)$-space $\la
W,n\ra$. 

Let $\alpha\in X$ and suppose that $t(\alpha)=i$. Since the images of
an element of $\Gamma$ under nontrivial elements of $N$ are pairwise
disjoint, no two elements of $\Gamma_\alpha$ of type $j>i$  lie in the
same $N$-orbit. However, if $i\geq 2$ then for $\beta\in X_1\cap
\Gamma_\alpha$ and $\gamma\in X_0\cap \Gamma_\alpha$ such that
$\gamma$ and $\beta$ are not incident, the translate of $\beta$ by
$\gamma$ lies both in $X_1\cap \Gamma_\alpha$ and the $N$-orbit of $\beta$. So the projection
$\pi_{/N}$ is not a covering. 
\end{example}

In Example \ref{eg:affine}, for each $i<d-1$ and $\alpha\in X$ of type
$i$, $\pi_{/N}$ induces an isomorphism from those elements of
$\Gamma_{\alpha}$ of type greater than $i$ to those elements of
$(\Gamma_{/N})_{\alpha^N}$ of type greater than $i$. Thus, given a flag $F_N$ of
$\Gamma_{/N}$ such that the least integer $i$ for which $F_N$ contains
an element of type $i$ is less than $d-1$, we deduce that there exists
a flag $F$ of $\Gamma$ such that $\pi_{/N}(F)=F_N$, and this also holds
trivially for flags $F_N$ of type $\{d-1\}$. This suggests the following result.
\begin{theorem}
Let $\Gamma=(X,*,t)$ be a geometry and let $\mathcal{B}$ be a
type-refining partition. Suppose that there is a total ordering $\leq$
on $I$ such that for each element $\alpha$ of $X$,
$\pi_{/\mathcal{B}}$ induces an isomorphism from $\{\beta\in
\Gamma_\alpha\mid t(\beta)\geq t(\alpha)\}$ onto
$\{B\in(\Gamma_{/\mathcal{B}})_{A}\mid t_{/\mathcal{B}}(B)\geq
t(\alpha)\}$, where $A$ is the unique element of $\mathcal{B}$
containing $\alpha$. Then \textsc{(FlagsLift)} holds. In particular,
$\Gamma_{/\mathcal{B}}$ is a geometry. 
\end{theorem}
\begin{proof}
Let $F_{\mathcal{B}}$ be a flag of $\Gamma_{/\mathcal{B}}$ and let $i$
be the least element of $t_{/\mathcal{B}}(F_{\mathcal{B}})$ according
to $\leq$. Let $A\in F_{\mathcal{B}}$ of type $i$ and $\alpha\in
A$. Then $\pi_{/\mathcal{B}}$ induces an isomorphism from
$Y=\{\beta\in \Gamma_\alpha\mid t(\beta)\geq t(\alpha)\}$ onto
$\{B\in(\Gamma_{/\mathcal{B}})_A\mid t_{/\mathcal{B}}(B)\geq
t(\alpha)\}=\pi_{/\mathcal{B}}(Y)$. Since $(F_{\mathcal{B}}\setminus
\{A\})\subseteq \pi_{/\mathcal{B}}(Y)$, it follows that there exists a
flag $F'\subseteq Y$ such that
$\pi_{/\mathcal{B}}(F')=F_{\mathcal{B}}\setminus \{A\}$. Thus
$F=F'\cup\{\alpha\}$ is a flag of $\Gamma$ which projects onto
$F_{\mathcal{B}}$ and so \textsc{(FlagsLift)} holds.  Hence by Lemma
\ref{lem:flagliftgeom}, $\Gamma_{/\mathcal{B}}$ is a geometry. 
\end{proof}

We also have the following general construction which lifts any geometry to a larger geometry having the original geometry as a quotient.

\begin{construction}
\label{con:blowup}
Let $\Gamma=(X,*,t)$ be a geometry of rank $n$ with
$G\leqslant\Aut(\Gamma)$, and let $\Delta$ be a graph with
$H\leqslant\Aut(\Delta)$ such that $H$ is vertex-transitive on $\Delta$. We form a new rank $n$ pregeometry
$\Gamma\times \Delta$ with element set $\{(\alpha,\delta)\mid
\alpha\in X,\delta\in V\Delta\}$ such that
$(\alpha,\delta)*(\beta,\delta')$ if and only if $\alpha*\beta$ in
$\Gamma$ and either $(\alpha,\delta)=(\beta,\delta')$, or $\alpha\neq \beta$ and $\delta\sim\delta'$ in $\Delta$. The type of
$(\alpha,\delta)$ is defined to be equal to the type of $\alpha$ in
the original geometry. The incidence graph of
$\Gamma\times \Delta$ is the incidence graph of $\Gamma$ with each
edge replaced by a copy of the standard double cover of $\Delta$. 
Note that $G\times H\leqslant\Aut(\Gamma\times \Delta)$ and contains
$1\times H$ as a normal subgroup. Moreover, $\Gamma$ is the quotient
of $\Gamma\times \Delta$ with respect to the orbits of $1\times H$.  
\end{construction}

\begin{theorem}
\label{thm:blowup}
Let $\Gamma \times \Delta$, $n$, $G$ and $H$ be as in Construction \ref{con:blowup}.
\begin{enumerate}
\item If $\Gamma$ and $\Delta$ are connected and $\Delta$ is not
  bipartite then $\Gamma\times \Delta$ is connected. 
\item $\Gamma\times \Delta$ is a cover of $\Gamma$ if and only if
  $\Delta$ is a matching. 
\item $\Gamma\times \Delta$ is a geometry if and only if each clique
  of $\Delta$ of size at most $n$ is contained in a clique of size
  $n$. 
\item $\Gamma\times \Delta$ is firm if and only if either
   \begin{enumerate}
      \item $\Gamma$ is firm and $\Delta$ contains a clique of size $n$, or
       \item every clique of size $r<n$ in $\Delta$ is contained in at
         least two cliques of size $n$. 
   \end{enumerate}
\item $G\times H$ acts flag-transitively on $\Gamma\times \Delta$ if
  and only if $G$ is flag-transitive on $\Gamma$ and for each $i\leq
  n$, $H$ acts transitively on the set of ordered cliques of $\Delta$ of size
  $i$. 
\end{enumerate}
\end{theorem}
\begin{proof}
The standard double cover of $\Delta$ is connected if and only if
$\Delta$ is connected and not bipartite (see for example \cite[p298]{GLP1}). Hence part (1) follows.
Given $(\alpha,\delta)\in \Gamma\times\Delta$, we have $(\Gamma\times\Delta)_{(\alpha,\delta)}=\{(\beta,\delta')\mid \alpha*\beta, \alpha\neq \beta,\delta\sim\delta'\}$ which has image in the quotient given by $\Gamma_{\alpha}=\{\beta\mid \beta*\alpha,\alpha\neq \beta\}$. This projection will be a bijection and also an isomorphism if and only if $\delta$ has a unique neighbour in $\Delta$, that is, if an only if $\Delta$ is a matching. Thus we have part (2).

Note that
$F=\{(\alpha_1,\delta_1),(\alpha_2,\delta_2),\ldots,(\alpha_r,\delta_r)\}$
is a flag of $\Gamma\times \Delta$ if and only if
$\{\alpha_1,\ldots,\alpha_r\}$ is a flag of $\Gamma$ of rank $r$ and
$\{\delta_1,\ldots,\delta_r\}$ is a clique $\Delta$ of size $r$.
Moreover, $F$ lifts to a chamber if and only if
$\{\delta_1,\ldots,\delta_r\}$ is contained in a clique of size $n$.
Thus the final three parts of the theorem follow. 
\end{proof}

\begin{remark}
Note that for $n>2$ in the above theorem, property~3 about the extension
of cliques in $\Delta$ excludes property~2. Hence the theorem does not
provide a way to generate proper covers of geometries of rank $>2$
which are geometries, which would imply that universal covers
(see~\cite{Pasi94} for example) do not exist. Even in rank~$2$, where
property~2 implies that $\Delta$ is bipartite, we get that the
resulting geometry is not connected. 
\end{remark}

\section{Quotients of Tits and Pasini}
\label{sec:survey}

There are at least four sources to consult when one wants to learn
about quotients of incidence geometries, namely
\cite{hbk,Gelbgras88,Pasi94,titslocal}. 

Pasini \cite[p244]{Pasi94} defines quotients of geometries according
to type-refining partitions. This is exactly the same as we do but one
has to keep in mind that Pasini's definition of a geometry is 
more restrictive than ours: he always assumes firmness and residual
connectivity. We say that the pregeometry $\Gamma=(X,*,t)$ is
\emph{residually connected}  if for each flag $F$ with
$|I\setminus t(F)|\geq 2$, the incidence graph $(X_F,*_F)$ of the
residue $\Gamma_F$ has nonempty vertex set and is connected.

The properties that Pasini studies are the following.
\begin{description}
\item[(PQ1)] Given a flag $F=\{\alpha_1,\alpha_2,\ldots ,\alpha_m\}$ in
$\Gamma$ and a class $B\in\mathcal{B}$ of type
$t_{/\mathcal{B}}(B)\not\in t(F)$, whenever there exist
$\beta_1,\beta_2,\ldots ,\beta_m$ in $B$ such that for each
$i\in\{1,2,\ldots ,m\}$ we have $\Gamma_{\beta_{i}}\cap
\pi_{/\mathcal{B}}(\alpha_i)\neq\varnothing$. Then there must exist an element $y$ in $B$ incident with
$F$. \item[(PQ2)] Every residue of rank $1$ meets at least two classes
of $\mathcal{B}$. \end{description}

The first property ensures that whenever the projection of a flag $F$
of $\Gamma$ can be extended by one element $B$ in the quotient
$\Gamma_{/\mathcal{B}}$, the flag $F$ is contained in a larger flag of
$\Gamma$  which projects onto $\pi_{/\mathcal{B}}(F)\cup\{B\}$.

\begin{lemma}\label{lem:PQ1impFL}
In a quotient of finite rank,
Pasini's axiom (PQ1) implies \textsc{(FlagsLift)}.
\end{lemma}
\begin{proof}
We know that flags of rank $<2$ lift. We now induct on the rank of the flag. Suppose $F_{\mathcal{B}}$ is a
flag of rank $k\geq 2$ in $\Gamma_{/\mathcal{B}}$. Choose a type in
$t(F_{\mathcal{B}})$ and an element $\alpha$ of this type such that
$B=\pi_{/\mathcal{B}}(\alpha)\in F_{\mathcal{B}}$. By the inductive
hypothesis, the flag
$F_{\mathcal{B}}'=F_{\mathcal{B}}\setminus\{B\}$ can be lifted
to a flag $F'$ which, by (PQ1), must be incident to some element
$\beta$ of $B$. The flag $F'\cup\{\beta\}$ projects to
$F_{\mathcal{B}}$.
\end{proof}

The converse of Lemma \ref{lem:PQ1impFL} is not
true. Figure~\ref{fig:flnotpq1} gives a small counterexample. The flag
$\{\alpha\}$ yields an incidence with the class $B$ in the quotient
given by the rectangles. There is no element of $B$ incident with
$\alpha$. Since the example has rank~2 we have \textsc{(Flagslift)}.

\begin{figure}
\caption{Counterexample to the converse of Lemma~\ref{lem:PQ1impFL}.}\label{fig:flnotpq1}
\begin{center}
\includegraphics[width=.3\textwidth]{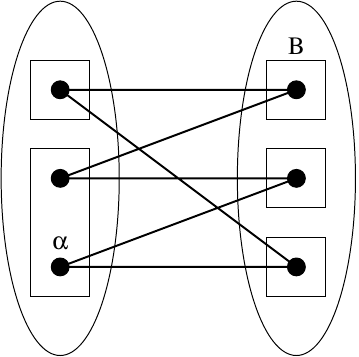}
\end{center}
\end{figure}

The stronger property of (PQ1) guarantees that $\pi_{/\mathcal{B}}$ is what Pasini calls
\emph{residually surjective}, that is to say, for each flag $F$ in
$\Gamma$,
$\pi_{/\mathcal{B}}(\Gamma_F)=(\Gamma_{/\mathcal{B}})_{\pi_{/\mathcal{B}}(F)}$.
The property (PQ1), together with (PQ2), which provides firmness of $\Gamma_{/\mathcal{B}}$, is
sufficient to show that the quotient $\Gamma_{/\mathcal{B}}$ is again a
firm and residually connected geometry.

It is also easy to notice that a residually surjective morphism of
geometries $\phi\colon\Gamma\rightarrow\Gamma'$, where both $\Gamma$
and $\Gamma'$ are firm and residually connected,  satisfies both
properties. In this case the second geometry is in fact (isomorphic to)
the quotient geometry of $\Gamma$ with respect to the partition
$\mathcal{B}=\{\phi^{-1}(\alpha')\mid \alpha'\in X'\}$ of $X$ into
\emph{fibres} for $\phi$.

In \cite{titslocal}, Tits uses the term `geometry' for what we call a
pregeometry. He defines quotients in the same way as we do but only
considers those partitions of $X$ that arise as the orbits of a
subgroup of automorphisms of the pregeometry. We use the term
\emph{orbit-quotient} for this kind of quotient to distinguish it from
the more general quotients defined with respect to an arbitrary type-refining
partition of $X$. We shall consider the properties of orbit-quotients as
well as the more general class of quotients of (pre)geometries.

Let $A$ be a group of automorphisms
of a pregeometry $\Gamma$ and consider the orbit-quotient $\Gamma_{/A}$ 
of $\Gamma$ with respect to the type-refining partition $\{\alpha^A\mid
\alpha\in X\}$.

For a flag $F$ of $\Gamma$, $A_F$ denotes the stabilizer of $F$ in $A$.
Note that the orbit set of $A_F$ in the residue $\Gamma_F$ is a
type-refining partition of $\Gamma_F$, inducing an orbit-quotient
$\left(\Gamma_F\right)_{/A_F}$ of $\Gamma_F$. Moreover, this partition
is a refinement of the partition of $\Gamma_F$ into $A$-orbits, and hence
$\pi_{/A}\left(\Gamma_F\right)$ can be viewed as an orbit-quotient of
$\left(\Gamma_F\right)_{/A_F}$. 

Tits uses the following properties that allow one to lift flags from the
quotient to the original geometry. 
\begin{description} \item[(TQ1)] For every flag $F$ in $\Gamma$,
the projection $\pi_{/A}\colon\Gamma\rightarrow\Gamma_{/A}\colon
\alpha\mapsto \alpha^A$ induces an isomorphism between the quotient
$\left(\Gamma_F\right)_{/A_F}$ and the residue
$\left(\Gamma_{/A}\right)_{\pi_{/A}(F)}$ of the projection $\pi_{/A}(F)$ of the flag
in the quotient. \item[(TQ2$'$)] The elements of an orbit of $A$ which
belong to a residue $\Gamma_F$ belong to a single orbit of $A_F$. 
\item[(TQ2$''$)] Suppose $\alpha\ast \beta$ in
$\Gamma$ and $F$ is a flag incident to some element of $\alpha^A$ and
some element of $\beta^A$. Then there exists $a\in A$ such that both
$\alpha^a$ and $\beta^a$ are incident to $F$. 
\item[(TQ3)] In the incidence graph of
$\Gamma$, the distance between two elements of the same orbit of $A$ is
at least $4$. 
\end{description}

We now prove a first result connecting Tits' axioms to Pasini's.
\begin{lemma}\label{lem:alice}
In an orbit-quotient of finite rank, Tits' axioms (TQ2$'$) and
(TQ2$''$) imply (PQ1).
\end{lemma}
\begin{proof}
We give a proof by induction on the rank of $\Gamma$. In  rank 1 the
lemma is clearly true. Assume the lemma holds in ranks smaller than
$m$.
Suppose we have a flag $F=\{\alpha_1,\alpha_2,\ldots ,\alpha_m\}$ in
$\Gamma$ and an orbit $\alpha^A$ of type not in $t(F)$ such that, for each
$i\in\{ 1,2,\ldots ,m\}$,  an element in the class of $\alpha_i$ is
incident to some element of $\alpha^A$. Since the class of $\alpha_i$ is an $A$-orbit for all $i$, there exist
$\alpha^{a_1},\alpha^{a_2},\ldots ,\alpha^{a_m}$  elements of $\alpha^A$ such that for each
$i\in\{1,2,\ldots ,m\}$ we have $\alpha_i*\alpha^{a_i}$. The induction
hypothesis provides an element $a\in A$ such that
$F\setminus\{\alpha_m\}$ is incident with $\alpha^a$.
Now $\alpha^{a_m}\ast \alpha_m$ and $F\setminus\{\alpha_m\}$ is a flag incident to some element of $(\alpha^{a_m})^A$, namely $\alpha^a$, and some element of $\alpha_m^A$, namely $\alpha_m$.
By (TQ2$''$), we can also find some $b\in A$ such that $\alpha_m^b$ and
$\alpha^{a_mb}$ are both incident with $F\setminus\{\alpha_m\}$. By
(TQ2$'$) we now have that both $\alpha_m$ and $\alpha_m^b$ are
in the same orbit of $A_{F\setminus\{\alpha_m\}}$. Hence we find $c$
in this stabilizer such that $\alpha_m^{bc}=\alpha_m$. Now the element
$\alpha^{a_mbc}$ is incident with the full flag $F$.
\end{proof}
\begin{corollary}
In an orbit-quotient of finite rank, Tits' axioms (TQ2$'$) and (TQ2$''$) imply \textsc{(FlagsLift)}.
\end{corollary}

We remark that Lemma \ref{lem:alice} is also useful in understanding
Gelbgras's proof of  (TQ2$'$) and (TQ2$''$) $\Rightarrow$
(TQ1), in particular the surjectivity of the map $\alpha^{A_F}\mapsto \alpha^A$.



\begin{lemma} \label{lem:TQ2'}
In an orbit-quotient of finite rank, Tits' axiom (TQ2$'$) is equivalent to
the fact that for any two flags $F$, $F'$ of $\Gamma$ such that $\pi_{/A}(F)=\pi_{/A}(F')$, there exists $a\in A$ mapping $F$ onto $F'$.
\end{lemma}
\begin{proof}
Assume first that  for any two flags $F$, $F'$ of $\Gamma$ such that $\pi_{/A}(F)=\pi_{/A}(F')$, there exists $a\in A$ mapping $F$ onto $F'$.
Let $\tilde{F}$ be a flag of $\Gamma$ and $\alpha,\alpha'$ be two elements of the residue $\Gamma_{\tilde{F}}$ in the same $A$-orbit. Then  $\pi_{/A}(\tilde{F}\cup\{\alpha\})=\pi_{/A}(\tilde{F}\cup\{\alpha'\})$. Thus there exists $a\in A$ mapping $\tilde{F}\cup\{\alpha\}$ onto $\tilde{F}\cup\{\alpha'\}$. As $A$ is type-preserving, $a$ stabilises $\tilde{F}$, and so  (TQ2$'$) holds.

Now assume  (TQ2$'$) holds and let  $F$, $F'$ be two flags of $\Gamma$ such that $\pi_{/A}(F)=\pi_{/A}(F')$. As the rank is finite, we can number the elements of the flags $F=\{\alpha_1,\alpha_2,\ldots ,\alpha_k\}$ and  $F=\{\alpha_1',\alpha_2',\ldots ,\alpha_k'\}$ such that $\alpha_i$ is in the same $A$-orbit as  $\alpha_i'$ for all $i$.
We will now show by induction that there exists an element of $A$ mapping $F$ to a flag containing $\{\alpha_1',\ldots,\alpha_i'\}$ for all $i\in \{1,\ldots,k\}$.
This is obvious for $i=1$, taking the element of $A$ which maps $\alpha_1$ to $\alpha_1'$.
Now suppose it is true for $i=j<k$ and let us prove it is true for $i=j+1$.
Let $a\in A$ mapping $F$ onto a flag containing $F_j:=\{\alpha_1',\ldots,\alpha_j'\}$, that is, $\alpha_\ell^a=\alpha_\ell'$ for all $\ell\leq j$. Now $\alpha_{j+1}'$ and $\alpha_{j+1}^{a}$ are both in the residue of $F_j$ and are in the same $A$-orbit. Therefore, by (TQ2$'$), there exists an element $b$ in $A_{F_j}$ mapping $\alpha_{j+1}^{a}$ onto $\alpha_{j+1}'$. Then $ab$ is an element of $A$ mapping $F$ onto a flag contanining $\{\alpha_1',\ldots,\alpha_{j+1}'\}$. This concludes the induction.
For $i=k$, we get an element of $A$ mapping $F$ onto $F'$.
\end{proof}

\begin{lemma}
An orbit-quotient of finite rank satisfying both \textsc{(FlagsLift)}
and  (TQ2$'$) also satisfies (TQ2$''$).
\end{lemma}
\begin{proof}
Suppose we have a flag $F$ of $\Gamma$ and two elements $\alpha$ and
$\beta$ which are incident and such that the residue $\Gamma_F$
intersects both orbits $\alpha^A$ and $\beta^A$. This means that
$\pi_{/A}(F)\cup\{\alpha^A,\beta^A\}$ is a flag in $\Gamma_{/A}$. This
flag must lift to a flag $\tilde{F}$ in $\Gamma$. Let us write
$\alpha^{a}$ and $\beta^{b}$ for the elements of
$\tilde{F}$ in $\alpha^A$ and $\beta^A$ respectively. Since
$\alpha^{a}$ and $\beta^{a}$ are incident (TQ2$'$)
provides $c\in A_{\alpha^{a}}$ such that
$\beta^{bc}=\beta^{a}$. The image $\tilde{F}^c$ is a flag containing $\alpha^a$ and $\beta^a$.
Now $\tilde{F}^c\setminus\{ \alpha^a,\beta^a\}$ and $F$ are two flags which $\pi_{/A}$ projects onto $\pi_{/A}(F)$. By Lemma \ref{lem:TQ2'}, there exists $d\in A$ mapping $\tilde{F}^c\setminus\{ \alpha^a,\beta^a\}$ onto $F$. Then $\tilde{F}^{cd}=F\cup\{\alpha^{ad},\beta^{ad}\}$, and (TQ2$''$) holds.
\end{proof}

Tits remarks that for finite rank, (TQ1) is equivalent to the
conjunction of (TQ2$'$) and 
(TQ2$''$) and that (TQ3) implies (TQ1). Detailed proofs of these
assertions, together with counterexamples for all other implications
can be found in \cite{Gelbgras88}. One of these counterexamples given
in~\cite{Gelbgras88} and discussed in the next paragraph is
interesting in that it illustrates a subtle 
difference between (TQ1) and the condition that $\pi_{/A}$ is
residually surjective.

We now summarize the above discussion in a proposition.
\begin{proposition}
Let $\Gamma$ be a geometry of finite rank with orbit-quotient $\Gamma_{/A}$.
\begin{enumerate}
\item {\rm (TQ1)} $\Leftrightarrow$ ({\rm (TQ2$'$)} and {\rm (TQ2$''$)})
\item {\rm (TQ3)} $\Rightarrow$ {\rm (TQ1)} $\Rightarrow$ {\rm (PQ1)} $\Rightarrow$  \textsc{(FlagsLift)}
\item \textsc{(FlagsLift)} and {\rm (TQ2$'$)} implies {\rm (TQ2$''$)}.
\end{enumerate}
\end{proposition}

We note that we showed in Lemma \ref{lem:iscovering} that for arbitrary quotients, if $\pi_{/\mathcal{B}}$ is surjective on corank 1 residues then (TQ3) implies that $\Gamma$ is a cover of $\Gamma_{/\mathcal{B}}$ and so by Lemma \ref{lem:coveringflags} \textsc{(FlagsLift)} holds. Lemma \ref{lem:3inj} showed that if (TQ3) holds and $\mathcal{B}$ is the set of orbits of some subgroup then $\pi_{/\mathcal{B}}$ is surjective on corank 1 residues.

\subsubsection*{Property (TQ1) versus residual surjectivity}
Let $\Gamma$ be the following rank 3 geometry: there are exactly two
elements of each type in $\Gamma$ and two elements are incident
whenever their types differ. This is a firm residually connected
geometry. Let us write the elements of type $i$ in this geometry as
$\alpha_i$ and $\alpha_i'$, for $i=0,1,2$. As a group $A$ of
automorphisms we take the group generated by the permutations
$(\alpha_0,\alpha_0')(\alpha_1,\alpha_1')$ and
$(\alpha_0,\alpha_0')(\alpha_2,\alpha_2')$. Looking at the residue of
the flag $F=\{\alpha_1,\alpha_2\}$ we see that (TQ2$'$) is not satisfied
since $A_F$ is trivial while $X_F=\{\alpha_0,\alpha_0'\}$. Hence (TQ1)
is also not satisfied. Indeed, we see that
$\left(\Gamma_{/A}\right)_{\pi_{/A}(F)}$ has just one element, whereas
$\left(\Gamma_F\right)_{/A_F}$ has two. Hence these two geometries
cannot be isomorphic. However, it should be remarked that when we
consider the partition $\mathcal{B}$ of the elements of $\Gamma$ into
orbits of $A$, the projection $\pi_{/\mathcal{B}}$ is residually
surjective. So there is a subtle difference  between (TQ1) and the
property of $\pi_{/\mathcal{B}}$ being residually surjective, the
latter one merely expressing the fact that taking residues and
projecting onto the quotient are commuting operations. We could view
(TQ1) as meaning that taking residues and taking quotients commute, if we
view the stabilizer $A_F$ of the flag as the natural automorphism group
with which to quotient the residue $\Gamma_F$. This again supports our decision to distinguish between general quotients and orbit-quotients.

\subsubsection*{Tits axioms and covers}
An important class of quotients is formed by coverings as defined in
Section \ref{sec:covers}.  
When $\mathcal{B}$ is the partition of the element set of a pregeometry
$\Gamma$ into singletons, then $\Gamma$ is trivially a cover of the
quotient $\Gamma_{/\mathcal{B}}$, taking $\pi_{/\mathcal{B}}$ as a
covering morphism.

In fact we saw in Lemma~\ref{lem:semireg} that in the case where an
orbit-quotient $\Gamma_{/A}$ is a cover, the stabiliser in $A$ of a
nonempty flag is always trivial. Thus 
(TQ1) is satisfied and also residual surjectivity holds in this case. 
 We note that the
two-dimensional version of Example~\ref{eg:affine} also appears
in~\cite{Gelbgras88}, as an example showing that (TQ3) is stronger than
(TQ1). 

\section{Inherited properties of quotients}
\label{sec:props}

\subsection{Connectivity of residues, truncations and quotients}

Let $J$ be a subset of the type set of a pregeometry $\Gamma=(X, *,
t)$. The \emph{$J$-truncation} of $\Gamma$ is the pregeometry
$\Gamma^J=(X^J, *^J, t^J)$ where $X^J=t^{-1}(J)$ and
incidence is inherited from $\Gamma$. The type function is the
restriction $t^J:X^J\rightarrow J$ of $t$. A path for which all
elements, except possibly the endpoints $p$ and $q$, are of type $i$
or $j$ is called an \emph{$\{i,j\}$-path}. 

We collect together the following results.
\begin{theorem}
\begin{enumerate}
\item \cite{BueCoh}
Let $\Gamma$ be a 
residually connected pregeometry with finite type set $I$, and let
$i,j\in I$ be distinct. Then for any $p,q\in X$, there is an
$\{i,j\}$-path from $p$ to $q$.  
\item \cite{BueCoh}
A residually connected pregeometry is a geometry if and only if no
flag of corank~$1$ is maximal.
\item \cite{BuSc84}
Let $\Gamma$ be a geometry over a finite set $I$. Then $\Gamma$ is
residually connected if and only if, for every two distinct types $i$
and $j$ and every flag $F$ of $\Gamma$ having no elements of type $i$
or $j$, the $\{i,j\}$-truncation of $\Gamma_F$ is
nonempty and connected.
\item All rank~$2$ truncations of a residually connected geometry are
connected.
\end{enumerate}
\end{theorem}

The geometry in Figure \ref{fig:conneg} is an example of a geometry
which is not residually connected but has all rank 2 truncations
connected. 

\begin{figure}
\begin{center}
\includegraphics{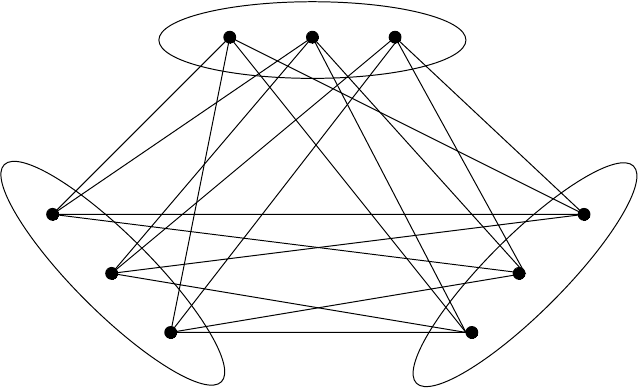}
\end{center}
\caption{A geometry
which is not residually connected, despite all rank 2 truncations being
connected}\label{fig:conneg}
\end{figure}

\begin{theorem}
\label{thm:connected}
Let $\Gamma=(X,*,t)$ be a pregeometry with type-refining partition
$\mathcal{B}$. \begin{enumerate} 
\item If $\Gamma$ is connected then $\Gamma_{/\mathcal{B}}$ is connected.
\item For each $J\subseteq I$, if the $J$-truncation $\Gamma^J$ is connected then also $(\Gamma_{/\mathcal{B}})^J=(\Gamma^J)_{/\mathcal{B}}$ is connected.
If each rank 2 truncation of $\Gamma$ is connected then each
  rank 2 truncation of $\Gamma_{/\mathcal{B}}$ is connected. 
\end{enumerate}
\end{theorem}
\begin{proof}
\begin{enumerate}
\item
Let $B,C\in\mathcal{B}$ and let $\alpha\in B$ and $\beta\in C$. As
$\Gamma$ is connected, there exists a path
$\alpha=\delta_0,\delta_1,\delta_2,\ldots,\delta_{l-1},\delta_l=\beta$
in the incidence graph of $\Gamma$. For each $\delta_i$, let $D_i$ be
the unique part of $\mathcal{B}$ containing $\delta_i$. Then
$B=D_0,D_1,\ldots,D_l=C$ is a path in the incidence graph of
$\Gamma_{/\mathcal{B}}$ and so $\Gamma_{/\mathcal{B}}$ is connected. 
\item
This follows from part (1) applied to the connected $J$-truncation of $\Gamma$. 
\end{enumerate}
\end{proof}

\subsection{Firmness}

Recall that
a geometry is said to be firm if every corank 1 flag is contained in
at least two chambers. Some sources include firmness in the definition
of a geometry, for example \cite{Pasi94}. 

One obvious way that firmness will not be preserved by quotients is if
there exists $i\in I$ such that $t^{-1}(i)\in \mathcal{B}$. Then the
quotient pregeometry $\Gamma_{/\mathcal{B}}$ only has one element of
type $i$ and so flags of type $I\setminus\{i\}$ in $\Gamma_{/\mathcal{B}}$ will not be contained in at least two chambers. One suggestion is that when quotienting we should ignore those
types which have a unique element. Alternatively, we may keep the
same rank and adjust the definition of firmness to require only that
flags are contained in at least two chambers whenever this is
feasible, that is, whenever $F$ is a flag and $i\in I\setminus t(F)$ such that there is more than one element of type $i$ in the geometry, the residue $\Gamma_F$ contains at least two elements of type $i$.

Even with this technicality, the quotient of a firm geometry (even a normal quotient) is not
necessarily firm as the following example shows. 
\begin{example}
\label{eg:notfirm}
Let $\Sigma$ be a complete multipartite graph with $m$ parts of size
$n$. For each integer $i$ with $1<i<n$, we define a rank 3 pregeometry $\Gamma^{(i)}$ whose elements are the vertices, edges and
$K_{i,i}$-subgraphs of $\Sigma$ with incidence being the inclusion
inherited from $\Sigma$. Each flag consisting of a vertex and an edge
is contained in $\binom{n-1}{i}^2$ chambers, each flag consisting of
an edge and a $K_{i,i}$ is contained in 2 chambers, while each flag
consisting of a vertex and a $K_{i,i}$ is contained in $i$
chambers. Hence $\Gamma^{(i)}$ is a firm geometry. Since $\Sigma$ is connected
the $\{\text{vertex,edge}\}$-truncation and the $\{\text{vertex,
}K_{i,i}\}$-truncation of $\Gamma^{(i)}$ are connected. However, the
$\{\text{edge, }K_{i,i}\}$-truncation is disconnected with
$\binom{m}{2}$ 
 connected components, each corresponding to a pair of partite
 blocks. The whole geometry $\Gamma^{(i)}$ is connected. 

The group $G=S_n\Wr S_m$ acts flag-transitively on $\Gamma^{(i)}$. Let
$N=S_n^m\norml G$. Then $N$ has $m$ orbits on vertices (the partite
blocks), $\binom{m}{2}$ orbits on edges (corresponding to pairs of
partite blocks) and $\binom{m}{2}$ orbits on $K_{i,i}$-subgraphs
(corresponding to pairs of partite blocks).  

In the quotient pregeometry $\Gamma_{/N}$, an incident (edge-orbit,
$K_{i,i}$-orbit) pair is contained in two chambers, an incident
(vertex-orbit, $K_{i,i}$-orbit) pair is contained in  one chamber and
an incident (vertex-orbit,edge-orbit) is contained in one
chamber. Hence $\Gamma_{/N}$ is a geometry but it is not firm. 
\end{example}

Example \ref{eg:notfirm} does not have connected rank 2 truncations. We have also examples of firm geometries with connected rank 2 truncations such that some quotient is a geometry but not firm.

\begin{example}
Let $I=\{1,2,3\}$ and $X=\{(i,j)\mid i,j\in I\}$. Define $t:X\rightarrow I$ by $t((i,j))=i$ and incidence by $(i,j)*(h,k)$ if and only if $i\neq h$ and $j\neq k$. Then the geometry $\Gamma=(X,*,t)$ is the complement of the $3\times 3$ grid and has connected rank 2 truncations. Let $\mathcal{B}=\{\{(i,1),(i,2)\}, \{(i,3)\}\mid i\in I\}$. Then in $\Gamma_{/\mathcal{B}}$, the flag $\{\{(1,1),(1,2)\},\{(2,3)\}\}$ is contained in a unique chamber, namely $\{\{(1,1),(1,2)\},\{(2,3)\},\{(3,1),(3,2)\}\}$.
\end{example}

\section{Group actions and coset pregeometries}
\label{sec:gractions}

Let $\Gamma$ be a pregeometry.  We refer to flags of type $J$ as
\emph{$J$-flags}. Given a group $G$ inducing automorphisms of
$\Gamma$, we say that $G$ is \emph{transitive on $J$-flags} if $G$
acts transitively on the set of all $J$-flags of $\Gamma$; and we say
that $G$ is \emph{flag-transitive} on  $\Gamma$ if $G$ is transitive
on $J$-flags for each $J\subseteq I$. 
Following \cite{gram}, we say that $G$ is \emph{incidence-transitive} on $\Gamma$ if for each
$J\subseteq I$ of size two, $G$ is transitive on $J$-flags. If for each $i\in I$, $G$ is transitive on the set $X_i$ of element of type $i$ we say that $G$ is \emph{vertex-transitive} on $\Gamma$.  Finally,
we say that $G$ is \emph{chamber-transitive} on $\Gamma$ if $G$ is
transitive on the set of all chambers. If $\Gamma$ is a geometry, since every flag is contained in a chamber, $G$ is flag-transitive on $\Gamma$ if and only if $G$ is
chamber-transitive on $\Gamma$. 

\remark
Given a partition $\mathcal{B}$ of $X$ and $G\leqslant\Aut(\Gamma)$,
$G$ induces automorphisms of $\Gamma_{/\mathcal{B}}$ if and only if
$\mathcal{B}$ is $G$-invariant. 
If $\mathcal{B}$ is the set of orbits of a subgroup $A$ of $G$, then
$\mathcal{B}$ is $G$-invariant if and only if the normal closure $N$
of $A$ in $G$ fixes each $A$-orbit setwise. In this case
$\Gamma_{/\mathcal{B}}$ is the normal quotient $\Gamma_{/N}$.

Given a pregeomety $\Gamma$ with type set $I$ and type-refining
partition $\mathcal{B}$, if $J\subseteq I$ we say that \emph{$J$-flags
  lift} if for each flag $F_{\mathcal{B}}$ of $\Gamma_{/\mathcal{B}}$
of type $J$, there exists a flag $F$ of $\Gamma$ such that
$\pi_\mathcal{B}(F)=F_{\mathcal{B}}$. 

We have the following theorem concerning flag-transitivity which implies Theorem \ref{thm:flagtrans}.

\begin{theorem}
\label{thm:ftrans}
Let $\Gamma=(X,*,t)$ be a pregeometry with type set $I$ and let
$J\subseteq I$ such that $\Gamma$ has $J$-flags. Suppose that
$G\leqslant \Aut(\Gamma)$ is transitive on $J$-flags and let
$\mathcal{B}$ be a type-refining partition preserved by $G$. Then $G$
is transitive on $J$-flags of $\Gamma_{/\mathcal{B}}$ if and only if
$J$-flags lift. 
\end{theorem}
\begin{proof}
Suppose first that for each $J$-flag $F_\mathcal{B}$ of
$\Gamma_{/\mathcal{B}}$, there exists a flag $F$ of $\Gamma$ such that
$\pi_\mathcal{B}(F)=F_\mathcal{B}$. Since $\mathcal{B}$ is
type-refining, $F$ also has type $J$. Then if $F_\mathcal{B}$ and
$F_\mathcal{B}'$ are $J$-flags of $\Gamma_{/\mathcal{B}}$, there
exist $J$-flags $F,F'$ of $\Gamma$ which project onto $F_\mathcal{B}$
and $F_\mathcal{B}'$ respectively. Since $G$ acts transitively on
$J$-flags of $\Gamma$, there exists $g\in G$ such that $F^g=F'$ and
hence $F_\mathcal{B}^g=F_\mathcal{B}'$. Thus $G$ is transitive on
$J$-flags of $\Gamma_{/\mathcal{B}}$. 

Conversely, suppose that $G$ is transitive on $J$-flags of
$\Gamma_{/\mathcal{B}}$. Now there exists a flag $F$ of $\Gamma$ of
type $J$ and  $\pi_\mathcal{B}(F)$ is a flag of
$\Gamma_{/\mathcal{B}}$ of type $J$. Given any flag $F_\mathcal{B}$ of
$\Gamma_{/\mathcal{B}}$ of type $J$, there exists  $g\in G$ such that
$\pi_\mathcal{B}(F)^g=F_\mathcal{B}$. Moreover, $F^g$ is a flag of
$\Gamma$ which projects onto $F_\mathcal{B}$. 
\end{proof}
\begin{corollary}
\label{cor:inctrans}
If $G$ is incidence-transitive on $\Gamma$ then $G$ is also
incidence-transitive on $\Gamma_{/\mathcal{B}}$. 
\end{corollary}
\begin{proof}
Follows from the fact that rank 2 flags in the quotient lift to
flags in the original pregeometry by definition. 
\end{proof}

Theorem~\ref{thm:flagtrans} is also a corollary of Theorem \ref{thm:ftrans}. The following example shows that Theorem \ref{thm:flagtrans} is not true in general for pregeometries.

\begin{example}
Let $\Gamma$ be the pregeometry with incidence graph  given in Figure
\ref{fig:quo3cycle} with type partition $\mathcal{B}=\{A,B,C\}$ and
let $A$ be the cyclic group of order two which induces a 2-cycle on both $A$, $B$ and $C$. Then the flag $\{A,B,C\}$ of $\Gamma_{/A}$ does
not lift while $A$ is flag-transitive on both $\Gamma$ and
$\Gamma_{/A}$. 
\end{example}

\begin{figure}
\begin{center}
\includegraphics[height=5cm]{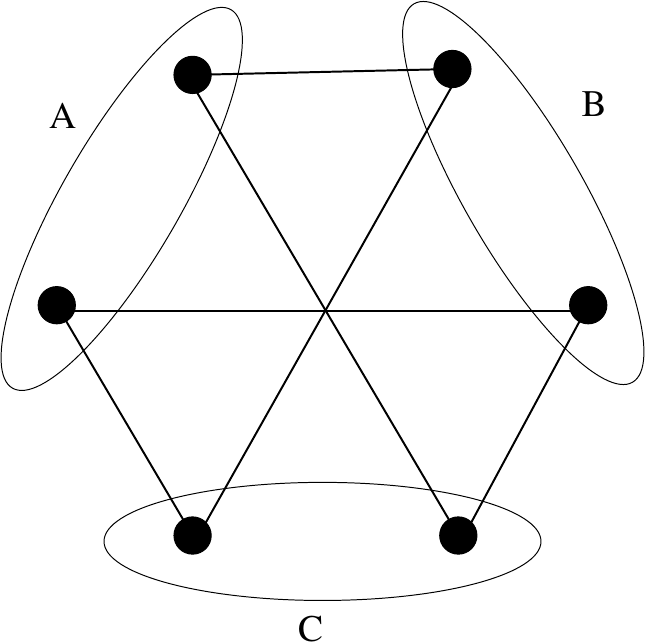}
\caption{A flag-transitive pregeometry with a quotient for which
  \textsc{(FlagsLift)} does not hold.} 
\label{fig:quo3cycle}
\end{center}
\end{figure}

We end this section by noting that for incidence-transitive
geometries, normal quotients have nice uniform combinatorial incidence
properties. In the rank 2 case it has been seen that such uniformity
implies that local symmetry properties such as primitivity and
2-transitivity are preserved \cite[Lemma 5.1]{GLP1}. These results
will then carry over to the arbitrary rank case by considering the
rank 2 truncations of the geometry.

\begin{lemma}
\label{lem:multicover}
Let $\Gamma$ be a pregeometry with normal quotient $\Gamma_{/N}$. Moreover, suppose that  $\Aut(\Gamma)$ is incidence-transitive and vertex-transitive on $\Gamma$. Then there exists an $I\times I$ array $K=(K_{ij})$ of integers such that if $\alpha\in X_i$ and $\beta\in X_j$ with $\alpha*\beta$, then $\alpha$ is incident with $K_{ij}$ elements of $\beta^N$.
\end{lemma}
\begin{proof}
Let $\gamma\in X_i,\delta\in X_j$ with $\gamma*\delta$ and let $\lambda$ be the number of elements of $\delta^N$ incident with $\gamma$. Now let $\alpha\in X_i$ and $\beta\in X_j$  such that $\alpha*\beta$. Since $\Aut(\Gamma)$ is incidence-transitive, there exists $g\in \Aut(\Gamma)$ such that $\gamma^g=\alpha$ and $\delta^g=\beta$. Moreover, $g$ maps $\gamma^N$ to $\alpha^N$ and $\delta^N$ to $\beta^N$. Hence $\alpha$ is adjacent to $\lambda$ elements of $\beta^N$.
\end{proof}

\subsection{Coset pregeometries}
\label{sec:coset}

We outline the theory of coset geometries, see for example \cite{hbk}.

Let $G$ be a group with subgroups $\{G_i\}_{i\in I}$. The \emph{coset
  pregeometry} $\Gamma=\Gamma(G,\{G_i\}_{i\in I})$ is the pregeometry
whose elements of type $i\in I$ are the right cosets of $G_i$ in $G$
and two cosets $G_ix$ and $G_jy$ are incident if and only if $G_ix\cap
G_jy$ is nonempty.  The rank 2 truncation with elements of 
type $i$ or $j$ is connected if and only if $\la G_i,G_j\ra=G$. 
The group $G$ acts by right multiplication on the elements of $\Gamma$
inducing a group of automorphisms.

For each $J\subseteq I$, the set $\{G_i\}_{i\in J}$ is a flag in
$\Gamma$ as each subgroup contains the identity element. Indeed, $G$
is flag-transitive on $\Gamma$ if and only if for each flag
$\{G_ix_i\}_{i\in J}$, for some subset $J$ of $I$, there is some
element $g$ lying in each $G_ix_i$. Note that $g$ then maps the flag
$\{G_i\}_{i\in J}$ to $\{G_ix_i\}_{i\in J}=\{G_ig\}_{i\in
  J}$. Moreover, since $\{G_i\}_{i\in J}$ is contained in the chamber
$\{G_i\}_{i\in I}$, if $\Gamma$ is flag-transitive then $\Gamma$ is a
geometry. If $I=\{1,2,3\}$, the flag-transitivity condition is equivalent to
$(G_1G_2)\cap (G_1G_3)=G_1(G_2\cap G_3)$ and also equivalent to
$(G_1\cap G_2)(G_1\cap G_3)=G_1\cap (G_2G_3)$ \cite[p79]{hbk}. 

We have the following characterisation of coset pregeometries which implies the well-known fact that  every flag-transitive geometry is isomorphic to a coset geometry. It also allows us to deduce that the quotient of a coset pregeometry is a coset pregeometry.

\begin{theorem}
\label{thm:cosetequiv}
Let $\Gamma$ be a pregeometry. Then there exists $G\leqslant\Aut(\Gamma)$ with subgroups $\{G_i\}_{i\in I}$ such that $\Gamma\cong\Gamma(G,\{G_i\}_{i\in I})$ if and only if $\Gamma$ contains a chamber and $\Aut(\Gamma)$ is incidence-transitive and vertex-transitive on $\Gamma$.
\end{theorem}
\begin{proof}
Suppose first that $\Gamma\cong \Gamma(G,\{G_i\}_{i\in I})$ for some $G\leqslant\Aut(\Gamma)$. Then for
each $i$, $G$ is transitive on the set of cosets of $G_i$ and hence is
transitive on the set of elements of type $i$. Moreover, if
$\{G_ix,G_jy\}$ is a flag of rank 2 then there exists $g\in G_ix\cap
G_jy$ and so $\{G_i,G_j\}^g=\{G_ix,G_jy\}$. Thus $G$ is incidence-transitive. Moreover,  $\{G_i\}_{i\in I}$ is a chamber as the identity
is contained in each $G_i$. 

Conversely, suppose that for each type $i$, $G\leqslant\Aut(\Gamma)$ is transitive on the set of elements of
type $i$, $G$ is incidence-transitive on $\Gamma$ and $\Gamma$ contains a chamber
$\{\alpha_i\}_{i\in I}$.  For each $i\in I$ let $G_i=G_{\alpha_i}$ and
if $\alpha$ is of type $i$ associate $\alpha$ with the coset $G_ix$,
where $x$ maps $\alpha_i$ to $\alpha$. Suppose $\beta*\alpha$ with
$\beta$ of type $j$, and suppose that $\beta$ is associated with
$G_jy$. Since $G$ is incidence-transitive there exists $g\in G$ such
that $(\alpha_i,\alpha_j)^g=(\alpha,\beta)$. Thus $g\in G_ix\cap
G_jy$ and so $G_ix$ is adjacent to $G_jy$ in $\Gamma(G,\{G_i\}_{i\in
  I})$.  Conversely, given two cosets $G_ix$, $G_jy$ with common
element $g$, then $G_ix=G_ig$ is associated with $\alpha_i^g$ and $G_jy=G_jg$ is
associated with $\alpha_j^g$ and since $\alpha_i$ is incident with
$\alpha_j$ it follows that $\alpha_i^g$ is incident with
$\alpha_j^g$. Thus $\Gamma\cong \Gamma(G,\{G_i\}_{i\in I})$. 
\end{proof}

\noindent \emph{Proof of Theorem \ref{thm:cosetquot}:}
This follows from Theorem \ref{thm:cosetequiv} as the projection $\pi_{/\mathcal{B}}$ maps chambers to chambers, and transitivity of $G$ on the set of elements of each type and incidence-transitivity is preserved. \hfill \qed \medskip

We now give the following example of a rank 4 flag-transitive coset geometry which has a normal quotient which is not a geometry.

\begin{example}
\label{eg:coseteg}
Let $A$ be an abelian group and $G=A^3$. We define four subgroups 
\begin{align*}
G_1 &=\{(x,1,x)\mid x\in A\}\\
G_2&=\{(x,1,1)\mid x\in A\}\\
G_3&=\{(x,x,1)\mid x\in A\}\\
G_4&=1
\end{align*}
and let $\Gamma=\Gamma(G,\{G_1,G_2,G_3,G_4\})$. The group $G$ acts on
the pregeometry preserving incidence and since $G_1\cap G_2\cap
G_3\cap G_4=\{1\}$, the action on flags of rank 4 is faithful.  
A sufficient condition for $G$ to be flag-transitive on $\Gamma$ is that the following five
conditions to hold (see \cite{deho94} or \cite[p79]{hbk}): 
$$(G_1G_2)\cap (G_1G_3)=G_1(G_2\cap G_3)$$
$$(G_1G_2)\cap (G_1G_4)=G_1(G_2\cap G_4)$$
$$(G_1G_3)\cap (G_1G_4)=G_1(G_3\cap G_4)$$
$$(G_2G_3)\cap (G_2G_4)=G_2(G_3\cap G_4)$$
$$(G_1G_2)\cap (G_1G_3)\cap (G_1G_4)=G_1(G_2\cap G_3\cap G_4).$$
Now 
$G_1G_2=\{(x,1,y)\mid x,y\in A\}$ and $G_1G_3=\{(xy,x,y)\mid x,y\in
A\}$. Thus $(G_1G_2)\cap (G_1G_3)=G_1$.  Also $G_2\cap G_3=1$ and so
$G_1(G_2\cap G_3)=G_1=(G_1G_2)\cap (G_2G_3)$. Since $G_4=1$ the next
three also hold. Finally, as $(G_1G_2)\cap (G_1G_3)=G_1$, $G_4=1$ and
$G_2\cap G_3=1$ the last condition holds. Thus $G$ is flag-transitive on $\Gamma$. Also note that $\la G_i,G_j\ra\neq G$
for any $i,j$ and so none of the rank 2 truncations is connected.  

Suppose now that $N=\{(x,x,x)\mid x\in A\}\norml G$. Then, for each
$i\in\{1,2,3,4\}$, $N$ acts intransitively on the elements of type
$i$. We examine the quotient pregeometry
$\Gamma_{/N}=\Gamma(G,\{NG_1,NG_2,NG_3,NG_4\})$.  
Let $a,b\in A$ with $a\neq b$ and let
$F=\{NG_1,NG_2,NG_3(a,b,b)\}$. Then $NG_1\cap NG_2=N$ and $(a,b,b)\in
(NG_2)\cap (NG_3(a,b,b))$. Also
$(ab^{-1},1,ab^{-1})=(ab^{-1},bb^{-1},bb^{-1}ab^{-1})\in G_1 \cap
(NG_3(a,b,b))$. Thus $F$ is a flag of $\Gamma_{/N}$.  For $\Gamma_{/N}$ to be a geometry, there
must be some coset of $N$ in $G$ which intersects nontrivially with
each of the elements of $F$. Since $(NG_1)\cap (NG_2)=N$ the only
possible candidate is $N$. However, $N\cap (NG_3(a,b,b))=\varnothing$
and so $F$ is not contained in a chamber. Thus $\Gamma_{/N}$ is not a
geometry. 
\end{example}
\begin{example}
\label{eg:cosetegb}
Let $\Sigma=\Gamma(G,\{G_1,G_2,G_3\})$ be the rank 3 truncation
of the geometry $\Gamma$ of the previous example. By Lemma
\ref{lem:rank3}, $\Sigma_{/N}$ is a geometry. Now 
$F=\{NG_1,NG_2,NG_3(a,b,b)\}$ is a flag of $\Sigma_{/N}$. Since $N\cap
NG_3(a,b,b)=\varnothing$, there is no element in common to all three
cosets in $F$. Thus $F$ does not lift to a flag of $\Sigma$. Moreover, Theorem \ref{thm:flagtrans} implies that $G$ is not flag-transitive on $\Sigma_{/N}$.
\end{example}

\begin{remark}
\label{rem:coseteg}
Note that Examples \ref{eg:coseteg} and \ref{eg:cosetegb} show that, if $G$ is flag-transitive on a  geometry $\Gamma$ and if $N\norml G$, then
\begin{enumerate}
\item[(a)] $\Gamma_{/N}$ need not be a geometry.
\item[(b)] Even if $\Gamma_{/N}$ is a geometry \textsc{(FlagsLift)} need not hold.
\end{enumerate}
\end{remark}

Both geometries in Examples \ref{eg:coseteg} and \ref{eg:cosetegb}
have disconnected rank 2 truncations. This suggests the following
question. 

\begin{question}
Does there exist a flag-transitive geometry with connected rank 2
truncations having a normal quotient which is not a geometry? 
\end{question}

\section{Diagrams}
\label{sec:diagram}

A rank 2 geometry whose incidence graph is complete bipartite, is
called a \emph{generalised digon}.

Let $\Gamma$ be a geometry with type set $I$. We define the
\emph{diagram} of $\Gamma$ to be the graph $(I,\sim)$ whose vertex
set is $I$ such that distinct types $i, j$ are adjacent if and only if there
exists a flag of cotype $\{i,j\}$ in $\Gamma$ whose residue is not a
generalised digon. This diagram ``without decorations'' is usually called the
\emph{basic diagram} (see~\cite{hbk,Pasi94}).

\remark\label{rem1}(a) The above definition does not exclude \emph{some} of
the residues of type $\{i,j\}$ from being generalised digons for an edge
$i\sim j$ in the diagram. Pasini~\cite{Pasi82} calls a geometry
\emph{pure} when $i\sim j$ implies that \emph{no} residue of type
$\{i,j\}$ is a generalised digon. For flag-transitive geometries all
residues of a given type are isomorphic and so there is no need to
make this distinction.

(b) If $i$ and $j$ are not adjacent in the diagram and $\alpha$ is an element of
type outside $\{i,j\}$, then $i$ and $j$ are also nonadjacent in the diagram of
the residue of $\alpha$. 

(c) In a pure geometry (b) is also true for adjacent types, that is, if $i$ and $j$ are adjacent and $\alpha$ is an element of
type outside $\{i,j\}$, then $i$ and $j$ are adjacent in the diagram of the residue of $\alpha$.


The following result is known as the Direct Sum Theorem (see \cite[p.81]{hbk} or \cite{Tits56} for an early version).
\begin{theorem}
\label{thm:directsum}
Let $\Gamma$ be a residually connected geometry of finite rank with
type set $I$ and let $i$ and $j$ be types in distinct connected
components of the diagram of $\Gamma$. Then every element of type $i$
in $\Gamma$ is incident with every element of type $j$.
\end{theorem}

Theorem \ref{thm:directsum} shows that a residually connected geometry
whose diagram is disconnected can easily be reconstructed from the
truncations corresponding to the connected components of its diagram.

\begin{lemma}
\label{lem:treech1}
In a residually connected geometry $\Gamma$ of finite rank whose
diagram $(I,\sim)$ 
contains no cycles, we have, for each path $i\sim j\sim k$ and
each choice $\alpha_i*\alpha_j*\alpha_k$ of elements of respective
types $i$, $j$, $k$, the incidence $\alpha_i*\alpha_k$. 
\end{lemma}
\begin{proof}
Note that $\alpha_i$ and $\alpha_k$ are in the residue of $\alpha_j$
and that $i$ and $k$ belong to different connected components of
$I\setminus\{j\}$. The Direct Sum Theorem then ensures the incidence
of $\alpha_i$ and $\alpha_k$. 
\end{proof}
We will call a geometry satisfying the conclusion of Lemma \ref{lem:treech1} \emph{$*$-transitive on paths}.

The converse of Lemma \ref{lem:treech1} is not necessarily true.  The Hoffman-Singleton graph $\HoSi$ is a regular graph of valency $7$, with $50$ vertices, and full automorphism group $\PSU(3,5):2$. 
Its maximal cocliques of size $15$ fall into two orbits under $\PSU(3,5)$.
Take $X_1$ and $X_{1}'$ to be two copies of the vertex set of $\HoSi$, and take $X_2$ and $X_{2}'$ to be two copies of an orbit of 15-cocliques under $\PSU(3,5)$.
For $a$, $a'$ in $X_1$, $X_{1}'$ respectively, and for  $b$, $b'$ in $X_2$, $X_{2}'$ respectively, we define incidence as follows:
\begin{align*}
a*a' &\iff \{a,a'\} \text{ is an edge of } \HoSi.\\
a*b &\iff a \in b\\
a*b' &\iff a \notin b'\\
a'*b &\iff a' \notin b\\
a'*b' &\iff a' \in b'\\
b*b' &\iff b \cap b'=\emptyset
\end{align*}
This rank 4 geometry is described in \cite{Neum} where it is seen that its diagram is the $4$-cycle $(X_1,X_1',X_2',X_2)$. The group $\PSU(3,5)$ is a flag-transitive group of automorphisms. It is very easily checked that for each path $i\sim j\sim k$ in the diagram (there are four of them) and
each choice $\alpha_i*\alpha_j*\alpha_k$ of elements of respective types $i$, $j$, $k$ we obtain $\alpha_i*\alpha_k$.


The diagram of a geometry which is $*$-transitive on paths is somewhat restricted in the sense that it
cannot have a triangle in its diagram.

\begin{theorem}
Let $\Gamma$ be a geometry and $i\sim j\sim k$ a path in the diagram
such that, for any path $\alpha_i*\alpha_j*\alpha_k$ in the
incidence graph, we have $\alpha_i*\alpha_k$. Then  $i\not\sim k$. 
\end{theorem}
\begin{proof}
Consider a flag $F$ of cotype $\{i,k\}$ with $i\sim j\sim k$ in the
diagram. Take any two elements
$\alpha_i,\alpha_k$ of respective types $i$ and $k$ in the residue of
$F$. Then we have $\alpha_i*\alpha_j*\alpha_k$, where $\alpha_j\in F$
with $t(\alpha_j)=j$, so that $\alpha_i*\alpha_k$. Hence the residue
of $F$ is a generalised digon. 
\end{proof} 

%

The quotient $\Gamma_{/\mathcal{B}}$ of a geometry $\Gamma$ does not
necessarily have the same diagram as $\Gamma$. However, if $\Gamma$ is
a 2-cover of $\Gamma_{/\mathcal{B}}$, then $\Gamma$ and
$\Gamma_{/\mathcal{B}}$ have the same diagram since their rank 2
residues are all isomorphic. 
\begin{example}
A cycle $(0,1,\ldots,7)$ of length~8 is a bipartite graph and hence is
the incidence graph of a rank 2
geometry. This geometry is firm and residually connected and obviously
has a connected diagram. The group $A\cong C_2$ generated by the map
$x\mapsto x+4$ acts on this geometry and has two orbits on the set of elements of each type. The quotient geometry with respect to these orbits has
incidence graph $K_{2,2}$ and hence a disconnected diagram.  
\end{example}

The next lemma discusses certain incidences between elements whose types lie in a tree structure $(J,E)$ imposed on a subset $J$ of types. Note that this tree is not necessarily related to the restriction of the basic diagram to $J$, and the lemma makes no assertions about incidence between $\alpha_i$ and $\alpha_j$ if $\{i,j\}\notin E$.    

\begin{lemma}
\label{lem:placingtree}
Let $\Gamma$ be a geometry of finite rank and let $A\leqslant \Aut(\Gamma)$.
Let $F_A=\{B_j\}_{j\in J}$ be a flag in $\Gamma_{/A}$ where
$t_{/A}(B_j)=j$ for all $j\in J$.  Let $E\subseteq J^{\{2\}}$
such that the graph $(J,E)$ is a tree, let $k\in J$ and let $\alpha_k\in B_k$. Then there exist elements $\{\alpha_j\}_{j\in J\setminus\{k\}}$ in $\Gamma$, with $\alpha_j\in B_j$ for all $j\in
J$, such that $\{i,j\}\in E$ implies $\alpha_i*\alpha_j$. 
\end{lemma}
\begin{proof}
If $F_A$ has rank 1 then the lemma is trivially true. Thus we
suppose that $F_A$ has rank at least 2. We start with any edge
$\{k,j\}$ in the tree $(J,E)$. Since $F_A$ is a flag there exist
$\beta_k\in B_k$ and $\beta_j\in B_j$ such that
$\beta_k*\beta_j$. There is an element $a\in A$ such that
$\beta_k^a=\alpha_k$ and we let $\alpha_j:=\beta_j^a\in B_j$ so that
$\alpha_j*\alpha_k$. Now let $J'=\{k,j\}$ and let $K'\subseteq J$ be
the set of types adjacent to a type in $J'$ in the tree $(J,E)$. Then
for each $i\in K'$ there exists a unique $\ell\in J'$ such that
$\{\ell,i\}\in E$. Moreover, as $F_A$ is a flag, arguing as before, there
exists $\alpha_i\in B_i$ incident with $\alpha_\ell$. We can then add all
types in $K'$ to $J'$ and repeat the process until $J'=J$.  
\end{proof}

We can now use the diagram to obtain sufficient conditions for some flags to lift.
\begin{corollary}
\label{cor:subtree}
Let $\Gamma$ be a residually connected geometry of finite rank whose diagram contains no cycles and let $A\leqslant \Aut(\Gamma)$.
Let $F_A=\{B_j\}_{j\in J}$ be a flag in $\Gamma_{/A}$ where $t_{/A}(B_j)=j$ for all $j\in J$. 
Suppose that the restriction to $J$ of the basic diagram of $\Gamma$ is a tree.
Then for each $k\in J$ and each $\alpha_k\in B_k$, there
exist elements $\{\alpha_j\}_{j\in J\setminus\{k\}}$ in $\Gamma$, with
$\alpha_j\in B_j$  for all $j\in J$, such that $\{\alpha_j\}_{j\in J}$
is a flag.
\end{corollary}
\begin{proof}
Let $E$ be the set of edges on $J$ induced by the diagram of $\Gamma$. By hypothesis, $(J,E)$ is a tree.
 By Lemma \ref{lem:placingtree}, we can find 
$\{\alpha_i\}_{i\in J}$ in $\Gamma$ with $\alpha_i\in B_i$ such that
if $\{i,j\}\in E$ then $\alpha_i*\alpha_j$. Now
consider $\ell,k\in J$. If  
$\{\ell,k\}\in E$  then
$\alpha_\ell*\alpha_k$. Otherwise, there exists a path $\ell=\ell_0,\ell_1,\ldots, \ell_m=k$ of length at least
$2$ in $(J,E)$, and so $\alpha_{\ell_i}*\alpha_{\ell_{i+1}}$ for $i=0,\ldots, m-1$. Thus we have $\alpha_\ell*\alpha_{\ell_1}*\alpha_{\ell_2}$, and so by Lemma
\ref{lem:treech1},
$\alpha_\ell*\alpha_{\ell_2}$. Working down the path between $\ell$ and $k$, we
can deduce that  
$\alpha_\ell*\alpha_k$. Hence $F=\{\alpha_i\}_{i\in J}$ is a flag.
\end{proof}

We give an example to show that Corollary \ref{cor:subtree} cannot be extended to arbitrary flags. In particular there may be a problem if the restriction of the diagram to $J$ is a forest and not a tree. However chambers do lift, and we prove this in Theorem \ref{thm:chamberlift}.

\begin{example}\label{D4example}{\rm
Let $Q$ be a hyperbolic quadric of equation $x_0x_1+x_2x_3+x_4x_5+x_6x_7=0$ in $PG(7,K)$, where $K$ is a field. Then the maximal totally isotropic subspaces are of (projective) dimension $3$, we call such subspaces {\it solids}. It is well-known that the solids can be partitioned into two sets $M^+$ and $M^-$ such that any two solids in the same set intersect in a subspace of even codimension (that is, a line or empty) and  two solids in different sets intersect in a subspace of odd codimension (that is, a plane or a point). Moreover, each totally isotropic plane is contained in exactly one solid of each set.
We now describe a rank 4 geometry $\Gamma$, whose elements are defined in the following way: 
\begin{center}
\begin{tabular}{|rc|rc|}
 \hline 
Type 1 :&  totally isotropic points& Type 3 :& solids in $M^+$\\
\hline
Type 2 :& totally isotropic lines& Type 4 :& solids in $M^-$\\
\hline
\end{tabular}
\end{center}
 Incidence is symmetrised inclusion, except that elements of type $3$ and $4$ are incident if they intersect in a plane.
It is well-known that  $\Gamma$ is a residually connected geometry (see \cite[p.35--36]{Pasi94} and recall that Pasini includes residual connectedness in his definition of a geometry). The geometry $\Gamma$ has diagram $D_4$, that is the edges of its basic diagram are $\{1,2\}$, $\{2,3\}$ and  $\{2,4\}$. 
Let $s_+$ be the solid of $M^+$ with equations  $0=x_0=x_2=x_4=x_6$ and $s_-$ be the solid of $M^-$ with equations  $0=x_1=x_2=x_4=x_6$.
 The map $\epsilon:[x_0,x_1,x_2,x_3,x_4,x_5,x_6,x_7]\mapsto [x_1,x_0,x_3,x_2,x_5,x_4,x_7,x_6]$ lies in $\Aut(\Gamma)$, since $\epsilon$ preserves $Q$, $M^+$, and $M^-$. Set $A:=\langle \epsilon\rangle$ and consider the triple $$F_A=(p_0^A,s_+^A,s_-^A),$$ where $p_0:=[1,0,0,0,0,0,0,0]$. 
We claim that $F_A$ is a flag of $\Gamma_{/A}$ of type $J=\{1,3,4\}$ which does not lift to a chamber of $\Gamma$. 
Indeed  $p_0*s_-$, $p_0^\epsilon= [0,1,0,0,0,0,0,0]*s_-$, and  $s_+*s_-$ since they intersect in the plane $\pi$ with equations $0=x_0=x_1=x_2=x_4=x_6$. 
However $\pi=s_+\cap s_-$ does not contain a point of $p_0^A=\{p_0,p_0^\epsilon\}$ and $s_+$ is not incident with $s_-^\epsilon$, and so $F_A$ does not lift to a chamber of $\Gamma$.
Moreover, it is obvious that $F_A$ is a maximal flag of  $\Gamma_{/A}$, and so  $\Gamma_{/A}$ is not a geometry.
}\end{example}

\begin{theorem}
\label{thm:chamberlift}
Let $\Gamma=(X,*,t)$ be a residually connected geometry of finite
rank whose diagram contains no cycles and let
$A\leqslant\Aut(\Gamma)$. If $F_A$ is a chamber of $\Gamma_{/A}$ then
there exists a chamber $F$ of $\Gamma$ such that $\pi_{/A}(F)=F_A$.  
\end{theorem}
\begin{proof}
Let $I$ be the set of types of $\Gamma$ and suppose that
$F_A=\{B_i\}_{i\in I}$. Let  
$J\subseteq I$ be a connected component of the diagram of
$\Gamma$. Then the restriction  
of the diagram to $J$ is a tree. By Corollary \ref{cor:subtree}, we can find 
$\{\alpha_i\}_{i\in J}$ in $\Gamma$ with $\alpha_i\in B_i$ such that
 $F=\{\alpha_i\}_{i\in J}$ is a flag. If
$J=I$ then $F$ is the  required chamber. If $J\neq I$, let $J_1,\ldots,J_r$ be the connected
components of the  
diagram and for each $J_s$, let $F_s$ be the flag of type $J_s$
constructed as above. By  
Theorem \ref{thm:directsum}, for two connected components $F_s,F_{s'}$
each element of  
$F_s$ is adjacent to every element of $F_{s'}$. Hence $F=F_1\cup
\ldots\cup F_r$ is the  
required chamber of $\Gamma$.
\end{proof}

The following corollaries follow immediately from Theorem \ref{thm:chamberlift} and Lemma \ref{lem:rank3}.
Note that we do not know in Corollary \ref{cor:7.10} if  $\Gamma_{/N}$ is a geometry. It would be interesting to find an example similar to Example \ref{D4example} in which $A$ is a normal subgroup of a chamber-transitive automorphism group.
\begin{corollary}\label{cor:7.10}
Let $\Gamma$ be a residually connected geometry of finite rank
whose diagram contains no cycles such that $G\leqslant\Aut(\Gamma)$ is
chamber-transitive on $\Gamma$ and $N\norml G$. Then $G$ is
chamber-transitive on $\Gamma_{/N}$. 
\end{corollary}

\begin{corollary}
Let $\Gamma$ be a residually connected rank $3$ geometry whose
diagram is not a cycle such that $G\leqslant\Aut(\Gamma)$ is
flag-transitive on $\Gamma$ and $N\norml G$. Then $\Gamma_{/N}$ is a
flag-transitive geometry. 
\end{corollary}

We saw in Example \ref{eg:cosetegb} a rank $3$ geometry $\Sigma$
with a group $G$ of automorphisms such that $G$ is flag-transitive on
$\Sigma$ but $G$ is not flag-transitive on the normal quotient
$\Sigma_{/N}$. However, since its rank $2$ truncations are not
connected $\Sigma$ is not residually connected.   

\begin{question}
Does there exist a residually connected geometry $\Gamma$ with group
of automorphisms $G$ such that $G$ is flag-transitive on $\Gamma$ but
not flag-transitive on some normal quotient $\Gamma_{/N}$? By
Corollary~\ref{cor:7.10} the basic diagram of such geometry $\Gamma$
must contain at least one cycle.
\end{question}

\section{Shadowable geometries}
\label{sec:shadowable}
Shadowable geometries are the kind of geometries that most readily come to mind. Elements of a chosen type are regarded as ``points'' and all other elements as subsets of points of various kinds. We show that orbit-quotients of shadowable geometries are geometries and that each flag-transitive shadowable geometry $\Gamma$ can be lifted to an arbitrarily large flag-transitive geometry $\hat{\Gamma}$ which admits $\Gamma$ as an orbit-quotient.

Shadowable geometries are embeddable in subset geometries, defined as follows.
Consider a finite set $S$ with $v>1$ elements and
an integer $k$ with $0<k<v$. Let $X_i$ denote the set of all
$(i+1)$-subsets of $S$ for $i\in I=\{0,1,\ldots ,k-1\}$. On the set
$X=\cup_{i\in I}X_i$ we define the relation $*$ to be (symmetrized)
inclusion and the type $t(\alpha)$ of an element $\alpha\in X$ is simply its
cardinality minus $1$. It is easy to verify that
$(X,*,t)$ is indeed a geometry. We call it the \emph{subset
  geometry} of rank $k$ on $v$ elements and denote it by $\ssg(v,k)$.

\remark The subset geometry $\ssg(v,k)$ is the $\{1,2,\ldots
,k\}$-truncation of the thin building of type $A_{v-1}$. In
particular, the basic diagram of $\ssg(v,k)$ is the path $(0,1,\ldots,k-1)$
of length $k-1$. 

For a type $i\in I$ we define the \emph{$i$-shadow} of an element
$\alpha$ of a geometry to be the set of all elements of type $i$ incident
with $\alpha$. The $i$-shadow of $\alpha$ is written $\sigma_i(\alpha)$. 

\begin{definition} 
A geometry $\Gamma = (X,*,t)$ with specified type $0\in I$ will be called
\emph{shadowable} provided the shadow operator $\sigma_0$ is a strong
embedding of $\Gamma$ into a subset geometry $\ssg(v,k-1)$ with
$|t^{-1}(0)|=v$. By strong embedding we mean an injective morphism of
geometries such that $\sigma_0(\alpha)*\sigma_0(\beta)$ if and only if
$\alpha *\beta$ for all $\alpha, \beta\in X$.
\end{definition}

\begin{lemma}
\label{lem:shadowable}
Let $\Gamma$ be a shadowable geometry and let $A\leqslant
\Aut(\Gamma)$. Then $\Gamma_{/A}$ is a geometry. Moreover, if $G$ is
flag-transitive on $\Gamma$ with $A\norml G$ then $G$ acts
flag-transitively on $\Gamma_{/A}$. 
\end{lemma}
\begin{proof}
Since $\Gamma$ is shadowable, there is an ordering $<$ on the types
such that given $i,j\in I$ we have $i<j$ if, given $\alpha$ of type
$i$ and $\beta$ of type $j$, $|\sigma_0(\alpha)|<|\sigma_0(\beta)|$. 
Let $F_A$ be a flag in $\Gamma_{/A}$ with type set
$\{i_1,\ldots,i_l\}$ and suppose that $i_1<i_2<\cdots<i_l$. This
ordering defines a tree on $\{i_1,\ldots,i_l\}$ and so by Lemma
\ref{lem:placingtree}, for each $B_i\in F_A$ of type $i$, there exists
$\beta_i\in B_i$ such that $\beta_1*\beta_2*\cdots*\beta_l$. Thus
$\sigma_0(\beta_1)\subset\sigma_0(\beta_2)\subset\cdots\subset\sigma_0(\beta_l)$
and hence $F=\{\beta_1,\ldots,\beta_l\}$ is a flag in $\Gamma$ which
projects onto $F_A$. Thus by Lemma \ref{lem:flagliftgeom},
$\Gamma_{/A}$ is a geometry. Moreover, if $G$ is flag-transitive on $\Gamma$ with $A\norml G$ then  
 Theorem \ref{thm:flagtrans} implies that $G$ is also flag-transitive
 on $\Gamma_{/A}$.  
\end{proof}

We have the following construction which allows any shadowable
geometry to be lifted to a larger geometry which has the original
geometry as a quotient. 

\begin{construction}
\label{con:liftshadowable}
Let $\Gamma$ be a shadowable geometry of rank $r$. Let $v=|t^{-1}(0)|$ and for each $i\in I\backslash\{0\}$, let
$\ell_i=|\sigma_0(\alpha)|$ for some element $\alpha$ of type $i$. Since $\sigma_0$ is a morphism, $\ell_i$ is independent of the
choice of $\alpha$.   

For each $n> 2$ and $j$ such that $1< j<n$, we define a new
geometry $\Gamma^{n,j}$ also of rank $r$. The elements of type $0$ are
the vertices of the complete multipartite graph $\Sigma$ with $v$
parts of size $n$ and we label the partite blocks of $\Sigma$ by the
elements of $\Gamma$ of type $0$. Then for each $i\in
I\backslash\{0\}$, the elements of type $i$ of $\Gamma^{n,j}$ are the
complete $\ell_i$-partite subgraphs of $\Sigma$ with blocks of size $j$
whose blocks are labelled by the elements of $\sigma_0(\alpha)$ for some
element $\alpha$ of $\Gamma$ of type $i$.  Incidence is given by
natural inclusion in the graph $\Sigma$. 
\end{construction}

\begin{theorem}
Let $\Gamma^{n,j}$ be obtained from Construction \ref{con:liftshadowable}.
\begin{enumerate}
\item $\Gamma^{n,j}$ is a geometry.
\item Let $H$ be a vertex-transitive group of automorphisms of $\Gamma$. Then $G=S_n\Wr H$ acts vertex-transitively on  $\Gamma^{n,j}$.
\item If $H$ is flag-transitive on $\Gamma$ then $G$ is
  flag-transitive on $\Gamma^{n,j}$. 
\item Let $N=S_n^v\norml G$. Then $\Gamma^{n,j}_{/N}\cong \Gamma$.
 \end{enumerate}
\end{theorem}
\begin{proof}
Since $\Sigma$ is a complete multipartite graph and flags of
$\Gamma^{n,j}$ are obtained from flags of $\Gamma$, it follows that
each flag of $\Gamma^{n,j}$ is contained in a chamber and hence
$\Gamma^{n,j}$ is a geometry. Clearly, $G=S_n\Wr H$ is a group of
automorphisms and the assertion about flag-transitivity follows. 
Let $N=S_n^v\norml G$. The orbits of $N$ on the elements of type $i$
correspond to the elements of $\Gamma$ of type $i$ and so
$\Gamma^{n,j}_{/N}\cong \Gamma$. 
\end{proof}

\end{document}